\documentclass[11pt,a4paper,leqno]{amsart}

\usepackage[latin1]{inputenc}
\usepackage[T1]{fontenc}
\usepackage{amsfonts}
\usepackage{amsmath}
\usepackage{amssymb}
\usepackage{eurosym}
\usepackage{mathrsfs}
\usepackage{palatino}
\usepackage{color}
\usepackage{xcolor}
\usepackage{esint}
\usepackage[
    colorlinks=true,
    linkcolor=blue,
    citecolor=blue,
    urlcolor=blue
]{hyperref}

\numberwithin{equation}{section}

\swapnumbers
\theoremstyle{plain}
\newtheorem{thm}[equation]{Theorem}
\newtheorem{lem}[equation]{Lemma}
\newtheorem{prop}[equation]{Proposition}
\newtheorem{cor}[equation]{Corollary}

\theoremstyle{definition}
\newtheorem{defn}[equation]{Definition}

\theoremstyle{remark}
\newtheorem{rem}[equation]{Remark}

\title{Sparse bounds for maximal rough singular integrals}

\author{Yuhao Wu}

\address{Center for Applied Mathematics, Tianjin University, Weijin Road 92, 300072 Tianjin, China}
\email{yuhao\_wu@tju.edu.cn}
\makeatletter
\@namedef{subjclassname@2020}{%
  \textup{2020} Mathematics Subject Classification}
\makeatother


\begin{document}

\allowdisplaybreaks

\begin{abstract}
Let $\Omega\in L^1(S^{d-1})$ have vanishing average, and let
$T_\Omega^\ast$ be the maximal truncation of the associated rough
homogeneous singular integral. We prove quantitative 
sparse bounds for $T_\Omega^\ast$. If
$\Omega\in L^\infty(S^{d-1})$, then, for every $1<p<\infty$,
\[
    \|T_\Omega^\ast\|_{(1,p)\text{-}\mathrm{sparse}}
    \lesssim_d
    p'\|\Omega\|_{L^\infty(S^{d-1})}.
\]
For unbounded angular kernels, if $1<q<\infty$ and
$\Omega\in L^{q,1}\log L(S^{d-1})$, then the same estimate holds
for $q'\leq p<\infty$, with the right-hand side replaced by
\[
    C_{d,q}p'
    \|\Omega\|_{L^{q,1}\log L(S^{d-1})}.
\]
These estimates retain a genuine $L^1$ average in the first
entry of the sparse form. In the bounded-kernel case,
the upper bound has the same linear growth in $p'$ as
the known sparse bound for the nonmaximal operator.
The unbounded-kernel estimate includes the critical
exponent $p=q'$.
As a consequence, we obtain weighted weak-type $(1,1)$
estimates for all $A_1$ weights in the bounded-kernel case
and for weights in $A_1\cap RH_{q'}$ in the unbounded-kernel case.
The proof combines physical-space linearization and
microlocal decomposition with localized sparse testing.
An amplitude decomposition of the second input, together
with the Rademacher--Menshov inequality, yields the
quantitative dependence on $p$.
\end{abstract}

\maketitle

\section{Introduction and main results}

Let \(d\geq 2\), and let \(\sigma\) denote the normalized surface measure
on \(S^{d-1}\). Given \(\Omega\in L^1(S^{d-1})\) with
\[
    \int_{S^{d-1}}\Omega(\theta)\,d\sigma(\theta)=0,
\]
consider the rough homogeneous singular integral
\[
    T_\Omega f(x)
    :=
    \operatorname{p.v.}\int_{\mathbb R^d}
        \frac{\Omega\bigl((x-y)/|x-y|\bigr)}{|x-y|^d}
        f(y)\,dy
\]
and its maximal truncation
\[
    T_\Omega^\ast f(x)
    :=
    \sup_{\varepsilon>0}
    \left|
        \int_{|x-y|>\varepsilon}
        \frac{\Omega\bigl((x-y)/|x-y|\bigr)}{|x-y|^d}
        f(y)\,dy
    \right|.
\]
The lack of angular regularity places these operators outside the
standard Calder\'on--Zygmund theory.  This is especially consequential
for \(T_\Omega^\ast\), where the roughness of the kernel must be handled
simultaneously with the supremum over truncation parameters.

The \(L^p\)-theory of rough homogeneous singular integrals originates
in the work of Calder\'on and Zygmund \cite{MR84633}.
Fourier-transform methods developed by Duoandikoetxea and Rubio de
Francia \cite{MR837527} provided a flexible treatment of both the
singular integral and its maximal truncation; see also
\cite{MR1647912}. The weighted theory was developed in
\cite{MR1047758,MR1089418}, while quantitative weighted estimates
were obtained in \cite{MR3625128} and further developed in
\cite{MR3969435}. For related developments concerning Herz spaces,
compositions of rough singular integrals, and non-standard rough
singular integrals, respectively, we refer to
\cite{MR1671956,MR4225825,MR4745059}.

The endpoint theory is substantially more delicate. For the
nonmaximal operator \(T_\Omega\), fundamental weak-type estimates were
obtained in \cite{MR951506,MR943929,MR938680}. Seeger \cite{MR1317232} 
proved that \(T_\Omega\) is of weak type \((1,1)\) in every dimension
\(d\geq2\) under the condition
\(\Omega\in L\log L(S^{d-1})\), using a microlocal decomposition of
the kernel. A general weak-type criterion for singular integrals with
rough kernels, together with several applications, was developed by
Ding and Lai \cite{MR3894030}. Related endpoint estimates on
homogeneous groups were obtained by Tao \cite{MR1757083}; see also
Seeger and Tao \cite{MR1839769} for sharp Lorentz-space estimates.

The corresponding endpoint problem for the maximal truncation
\(T_\Omega^\ast\) remained open much longer. Honz\'ik
\cite{MR4165473} and Bhojak and Mohanty
\cite{MR4554742} obtained endpoint estimates with logarithmic
bumps. Lai \cite{lai2025weak11estimatemaximal} subsequently proved
that \(T_\Omega^\ast\) is of weak type \((1,1)\) whenever
\(\Omega\in L\log L(S^{d-1})\). Under the same \(L\log L\) assumption, Bhojak and Shrivastava
\cite{bhojak2026endpointvariationjumpinequalities} established
weak-type estimates for the associated variation and jump
operators; their variational estimate in particular recovers the
weak-type estimate for \(T_\Omega^\ast\). For angular kernels in
block spaces, Liu, Liu, and Zhang \cite{MR5099059} obtained an
\(L\log\log L\)-type endpoint estimate for \(T_\Omega^\ast\).
Related endpoint results for the rough maximal operator and the
maximal Calder\'on commutator with rough kernel were obtained in
\cite{MR4861142,MR4906326}.

Sparse domination provides a localized strengthening of norm
inequalities. For the nonmaximal operator \(T_\Omega\),
Conde-Alonso, Culiuc, Di Plinio, and Ou \cite{MR3668591} proved an
 \((1,p)\)-sparse bound. For bounded angular kernels their
estimate holds for every \(p>1\), whereas for
\(\Omega\in L^{q,1}\log L(S^{d-1})\) it holds in the range
\(p\geq q'\). See also
\cite{MR3085756,MR3484688,MR4007575,MR4018107} for related sparse
methods.

For the maximal truncation, the previously known estimates have a
different averaging structure. Di Plinio, Hyt\"onen, and Li
\cite{MR4245601} proved, for
\(\Omega\in L^\infty(S^{d-1})\), the symmetric estimate
\[
    \sup_{0<\varepsilon<1}
    \varepsilon
    \|T_\Omega^\ast\|_{(1+\varepsilon,1+\varepsilon)
                    \text{-}\mathrm{sparse}}
    \lesssim_d
    \|\Omega\|_{L^\infty(S^{d-1})}.
\]

Tao and Hu \cite{MR4728821} subsequently obtained, for every
\(1<r<\infty\), a refined sparse estimate consisting of a
\((1,r)\)-term, with coefficient \(r'\), and an additional
\((L^\Phi,L^r)\)-term, where
\[
    \Phi(t)=t\log\log(e^2+t),
\]
together with corresponding quantitative weighted consequences.
Choudhary, Shrivastava, and Shuin
\cite{MR4882780} established symmetric sparse bounds
for maximal oscillatory rough singular integrals. Sparse methods for
commutators of rough singular integrals were developed by Lan, Tao,
and Hu \cite{MR4168195}.

The maximal sparse estimates above involve either
$L^{1+\varepsilon}$-averages in both entries or an additional
Orlicz average in the first entry. In this paper, we establish
a pure $(1,p)$-sparse bound for $T_\Omega^\ast$.
For bounded angular kernels, the estimate holds for every
$p>1$, with a constant growing at most linearly in $p'$.
We also treat unbounded angular kernels in
$L^{q,1}\log L(S^{d-1})$ throughout the range
$q'\leq p<\infty$, including the critical exponent $p=q'$.

For a cube \(Q\subset\mathbb R^d\), set
\[
    \langle f\rangle_{p,Q}
    :=
    \left(\frac{1}{|Q|}\int_Q|f|^p\right)^{1/p},
    \qquad
    \langle f\rangle_Q:=\langle f\rangle_{1,Q}.
\]
Recall that a collection \(\mathcal S\) of cubes is called
\(\eta\)-sparse, \(0<\eta<1\), if there exist measurable sets
\(E_Q\subset Q\), \(Q\in\mathcal S\), such that
\[
    |E_Q|\geq\eta|Q|
    \quad\text{and}\quad
    E_Q\cap E_{Q'}=\emptyset
    \quad\text{whenever }Q\neq Q'.
\]
The Lorentz--Zygmund functional used in the second part of the theorem
is defined in Section~\ref{sec:preliminaries}.

\begin{thm}\label{thm:main}
Let \(d\geq1\), and suppose that
\(\Omega\in L^1(S^{d-1})\) satisfies
\[
    \int_{S^{d-1}}\Omega(\theta)\,d\sigma(\theta)=0.
\]

\begin{enumerate}
\item
If \(\Omega\in L^\infty(S^{d-1})\), then, for every
\(1<p<\infty\) and every pair of bounded, compactly supported
functions \(f_1,f_2\), there exists a \(1/2\)-sparse collection
\(\mathcal S=\mathcal S(f_1,f_2,p)\) such that
\begin{equation}\label{eq:main-Linfty}
    \left|
        \left\langle T_\Omega^\ast f_1,f_2\right\rangle
    \right|
    \lesssim_d
    p'\|\Omega\|_{L^\infty(S^{d-1})}
    \sum_{Q\in\mathcal S}
        |Q|\langle f_1\rangle_Q\langle f_2\rangle_{p,Q}.
\end{equation}

\item
Let \(1<q<\infty\). If
\(\Omega\in L^{q,1}\log L(S^{d-1})\), then, for every \(q'\leq p<\infty\)
and every pair of bounded, compactly supported functions
\(f_1,f_2\), there exists a \(1/2\)-sparse collection
\(\mathcal S=\mathcal S(f_1,f_2,p)\) such that
\begin{equation}\label{eq:main-LqlogL}
    \left|
        \left\langle T_\Omega^\ast f_1,f_2\right\rangle
    \right|
    \lesssim_{d,q}
    p'\|\Omega\|_{L^{q,1}\log L(S^{d-1})}
    \sum_{Q\in\mathcal S}
        |Q|\langle f_1\rangle_Q\langle f_2\rangle_{p,Q}.
\end{equation}
\end{enumerate}
\end{thm}

Theorem~\ref{thm:main} provides a maximal counterpart of the
sparse bounds in \cite{MR3668591}, retaining an \(L^1\)
average in the first entry and, for bounded angular kernels,
the same linear upper bound in \(p'\).
We do not claim that this dependence is optimal.
For unbounded angular kernels, the endpoint \(p=q'\)
leads to the reverse H\"older condition \(RH_{q'}\)
in the weighted consequence below.

The proof builds on Lai's physical-space approach to the
weak-type $(1,1)$ estimate
\cite{lai2025weak11estimatemaximal}.
The passage to sparse domination requires additional
testing estimates localized to stopping regions.
These estimates must preserve the cancellation of the
stopping atoms at separated scales and remain uniform
over the measurable truncation parameters.
Establishing this localization, together with the
quantitative dependence on $p$, is the main task of the paper.

We record the following weighted endpoint consequence. As usual,
\(L^{1,\infty}(w)\) denotes the weak space with respect to the measure
\(w(x)\,dx\), while \(A_1\) and \(RH_s\) denote the usual Muckenhoupt
and reverse H\"older classes.

\begin{cor}[Weighted weak-type endpoint]
\label{cor:weighted-weak-endpoint}
Suppose that \(\Omega\) has vanishing average.
\begin{enumerate}
\item
If \(\Omega\in L^\infty(S^{d-1})\), then
\[
    T_\Omega^\ast:
    L^1(w)\longrightarrow L^{1,\infty}(w)
\]
for every \(w\in A_1\).

\item
Let \(1<q<\infty\) and
\(\Omega\in L^{q,1}\log L(S^{d-1})\). Then
\[
    T_\Omega^\ast:
    L^1(w)\longrightarrow L^{1,\infty}(w)
\]
for every \(w\in A_1\cap RH_{q'}\).
\end{enumerate}
\end{cor}

\begin{proof}
Fix $\rho>1$ for which Theorem~\ref{thm:main} provides a
$(1,\rho)$-sparse bound. In the notation of \cite{MR3897012},
this means that $T_\Omega^\ast\in S(1,\rho')$.
If necessary, the sparse form may be transferred to finitely many
adjacent dyadic grids, changing only the dimensional constant.
Apply \cite[Theorem~1.4]{MR3897012} with
\[
    p_0=1,\qquad q_0=\rho',\qquad
    r=\frac{1+\rho'}2\in(1,\rho').
\]
Since $(q_0/p_0)'=\rho$, it follows that
\[
    \|T_\Omega^\ast f\|_{L^{1,\infty}(w)}
    \leq C_{\Omega,w,\rho}\|f\|_{L^1(w)}
    \qquad\text{for }w\in A_1\cap RH_\rho.
\]
In the unbounded-kernel case, choose $\rho=q'$.
The endpoint $p=q'$ is included in Theorem~\ref{thm:main},
so this gives the second assertion.

In the bounded-kernel case, every $w\in A_1$ belongs to
$RH_{1+\varepsilon_w}$ for some $\varepsilon_w>0$;
see \cite{MR2990061}. Choose $\rho=1+\varepsilon_w$ and use
the first part of Theorem~\ref{thm:main}. This proves the first assertion.
The estimates, initially obtained on bounded, compactly supported
functions, extend to $L^1(w)$ by density and sublinearity.
\end{proof}

Taking \(w\equiv1\) recovers the corresponding unweighted weak-type
\((1,1)\) estimates.
In dimension
\(d=1\), cancellation on \(S^0=\{-1,1\}\) makes \(T_\Omega^\ast\) a
constant multiple of the maximal Hilbert transform, so the theorem
follows from the usual sparse domination for maximal truncations of
Calder\'on--Zygmund operators; see, for instance, \cite{MR3484688}.
The proof below therefore assumes \(d\geq2\).

\medskip
\noindent\textit{Outline of the proof.}
We use the abstract sparse domination principle of Di Plinio,
Hyt\"onen, and Li \cite{MR4245601}, which reduces the theorem to two
localized testing estimates for stopping collections. 

Our starting point is Lai's physical-space organization of the dyadic
kernel pieces \cite{lai2025weak11estimatemaximal}. The main task is to
adapt this global endpoint argument to the localized testing forms
required by the sparse recursion. We localize the kernel pieces to
cubes of the corresponding scale and sort the cubes according to
their occupancy. Along each resulting nested chain, the maximal
truncation is reduced to a maximal partial sum, to which the
Rademacher--Menshov inequality is applied. For $s>1$, the physical localization preserves the
cancellation of each stopping atom. The adjacent case
$s=1$ is treated separately using kernel size estimates,
without requiring cancellation of the truncated atoms.

The interaction between the kernel scale and the smaller atomic scale
is treated by a microlocal decomposition adapted from
Seeger~\cite{MR1317232} and Lai~\cite{lai2025weak11estimatemaximal}.
The endpoint estimates give decay for signed sums in $L^2$.
We apply Rademacher--Menshov at exponent $2$.
For the quantitative testing estimate, we split the second
input according to its amplitude and combine the maximal
$L^2$ decay with a positive-kernel $\mathcal Y_1$ pairing
bound. This gives a fixed-separation paired estimate of
size $2^{-c_ds/p'}$, whose sum over $s$ is bounded by $C_dp'$.

For $\Omega\in L^{q,1}\log L(S^{d-1})$, a scale-dependent
angular decomposition gives the first localized estimate for every
$p\geq q'$. The second testing estimate follows directly from a
localized $L^r$ bound and the classical maximal estimate for $L^q$
angular kernels. Finally, the difference between radial and dyadic
truncations is controlled by the rough maximal operator. Its sparse
bound follows from the dyadic theorem by separating the constant
and mean-zero parts of $|\Omega|$.

The paper is organized as follows.  Section~\ref{sec:preliminaries}
contains the dyadic, sparse, stopping-time, Lorentz--Zygmund, and
Rademacher--Menshov preliminaries.  Section~\ref{sec:sparse-domination}
proves the bounded-kernel part of Theorem~\ref{thm:main}.
Section~\ref{sec:Lq1logL-extension} treats unbounded angular kernels
and completes the proof.

\section{Preliminaries}\label{sec:preliminaries}

In this section, we fix the notation used throughout the paper and collect
the auxiliary notions needed in the sparse domination argument. We first
recall sparse forms, translated cube families, and the Lorentz--Zygmund functional
appearing in the main theorem. We then introduce stopping collections and
their associated localized spaces. The section concludes with the
Rademacher--Menshov inequality used to control maximal partial sums.

\subsection{Notation, sparse forms, and the Lorentz--Zygmund functional}
\label{subsec:notation-sparse}

Throughout the paper, the surface measure \(\sigma\) on
\( S^{d-1}\) is normalized so that
\[
    \sigma( S^{d-1})=1.
\]
For a measurable set \(E\subset \mathbb R^d\), we write \(\mathbf 1_E\) for
its indicator function and \(|E|\) for its Lebesgue measure.
Given a cube
\(Q\subset\mathbb R^d\), let \(\ell(Q)\) denote its side length. For
\(\lambda>0\), \(\lambda Q\) denotes the cube concentric with \(Q\) whose
side length is \(\lambda\ell(Q)\). If \(Q\) is dyadic, its dyadic scale
\(s_Q\in\mathbb Z\) is determined by
\[
    \ell(Q)=2^{s_Q}.
\]

For \(0<p<\infty\), we retain the notation
\(\langle f\rangle_{p,Q}\) and \(\langle f\rangle_Q\) introduced
before Theorem~\ref{thm:main}, with the usual essential-supremum
interpretation when \(p=\infty\).
For \(1\leq p<\infty\), define
\[
    M_p f(x):=\sup_{Q\ni x}\langle f\rangle_{p,Q}
             =\bigl(M(|f|^p)(x)\bigr)^{1/p},
\]
where \(M=M_1\) is the Hardy--Littlewood maximal operator and the supremum
is taken over all cubes containing \(x\). If \(1\leq p\leq\infty\), then
\(p'\) denotes the conjugate exponent. We use the bilinear pairing
$\langle f,g\rangle:=\int_{\mathbb R^d}f(x)g(x)\,dx$,
also for complex-valued functions.
We use the Fourier transform
convention
\[
    \widehat f(\xi)
    =\int_{\mathbb R^d} e^{-2\pi i x\cdot\xi}f(x)\,dx.
\]
The relation \(A\lesssim B\) means that \(A\leq CB\) for a constant \(C\)
depending only on the dimension and on fixed structural parameters.
Dependence on an additional parameter \(\alpha\) is indicated by
\(A\lesssim_\alpha B\).
Throughout, \(|x|\) denotes the Euclidean norm on \(\mathbb R^d\),
whereas \(|x|_\infty\) denotes the \(\ell^\infty\)-norm. 
If $A,B\subset\mathbb R^d$, define
\[
    \operatorname{dist}(A,B)
    :=
    \inf_{\substack{x\in A\\y\in B}}|x-y|,
    \qquad
    \operatorname{dist}_\infty(A,B)
    :=
    \inf_{\substack{x\in A\\y\in B}}|x-y|_\infty.
\]
When one of the sets is a singleton, we identify the point with that
singleton.

For \(1\leq p_1,p_2<\infty\), the associated sparse form is
\[
    \Lambda_{\mathcal S,p_1,p_2}(f_1,f_2)
    :=\sum_{Q\in\mathcal S}
      |Q|\,\langle f_1\rangle_{p_1,Q}
            \langle f_2\rangle_{p_2,Q}.
\]
If \(T\) is a sublinear operator initially defined on bounded,
compactly supported functions, its sparse \((p_1,p_2)\)-norm,
\(\|T\|_{(p_1,p_2)\text{-}\mathrm{sparse}}\), is the least constant \(C\)
with the following property: for every pair of bounded, compactly
supported functions \(f_1,f_2\), there exists a \(1/2\)-sparse collection
\(\mathcal S=\mathcal S(f_1,f_2)\) such that
\[
    |\langle Tf_1,f_2\rangle|
    \leq C\,\Lambda_{\mathcal S,p_1,p_2}(f_1,f_2).
\]

We also record the precise Lorentz--Zygmund functional used below. Let
\(1<q<\infty\), let \(\Omega\) be measurable on \(S^{d-1}\), and
define its distribution function by
\[
    \mu_\Omega(t)
    :=\sigma\bigl(\{\theta\in S^{d-1}:
                         |\Omega(\theta)|>t\}\bigr),
    \qquad t>0.
\]
Set
\[
    L_q(\Omega)
    :=q\int_0^\infty \mu_\Omega(t)^{1/q}\,dt
     =\|\Omega\|_{L^{q,1}(S^{d-1})}.
\]
For $0<L_q(\Omega)<\infty$, define
\begin{equation}\label{eq:Lq1logL-functional}
    \|\Omega\|_{L^{q,1}\log L(S^{d-1})}
    :=q\int_0^\infty
       \log\left(e+\frac{t}{L_q(\Omega)}\right)
       \mu_\Omega(t)^{1/q}\,dt.
\end{equation}
Set this quantity equal to zero when $\Omega=0$ almost
everywhere, and to $+\infty$ when $L_q(\Omega)=+\infty$.
We write $\Omega\in L^{q,1}\log L(S^{d-1})$ when this
functional is finite.

Let \(\Omega^\ast\) denote the decreasing rearrangement of
\(|\Omega|\) on \((0,1)\). For \(0<L_q(\Omega)<\infty\), layer cake
and the standard rearrangement characterization give
\[
\begin{aligned}
    \|\Omega\|_{L^{q,1}\log L(S^{d-1})}
    &\simeq
    \int_0^1 t^{1/q-1}\Omega^\ast(t)
       \log\left(e+\frac{\Omega^\ast(t)}{L_q(\Omega)}\right)\,dt\\
    &\simeq_q
    \int_0^1 t^{1/q-1}\log(e/t)\Omega^\ast(t)\,dt.
\end{aligned}
\]
Indeed, \(qt^{1/q}\Omega^\ast(t)\leq L_q(\Omega)\) gives one
direction of the second comparison. For the reverse direction,
split according to
\(\Omega^\ast(t)>L_q(\Omega)t^{-1/(2q)}\).
On the complementary set, the last integral is bounded by
\[
    L_q(\Omega)\int_0^1t^{1/(2q)-1}\log(e/t)\,dt
    \lesssim_q L_q(\Omega)
    \leq \|\Omega\|_{L^{q,1}\log L(S^{d-1})}.
\]
In particular, suppose that
\(\Omega\in L^{q,1}\log L(S^{d-1})\), and set
\[
    \Omega_N^{\mathrm{tail}}
    :=\Omega\mathbf1_{\{|\Omega|>N\}}.
\]
Since
\[
    \sigma\bigl(\{|\Omega_N^{\mathrm{tail}}|>0\}\bigr)
    =\mu_\Omega(N)\longrightarrow0,
\]
we have, for every \(0<t<1\),
\[
    (\Omega_N^{\mathrm{tail}})^\ast(t)
    \leq\Omega^\ast(t),
    \qquad
    (\Omega_N^{\mathrm{tail}})^\ast(t)\longrightarrow0.
\]
By dominated convergence and the preceding norm equivalence,
\[
    \|\Omega_N^{\mathrm{tail}}\|_{L^{q,1}\log L(S^{d-1})}
    \longrightarrow0
    \qquad(N\to\infty).
\]
\subsection{Translated cube families}\label{subsec:shifted-dyadic-grids}

Let \(\mathcal D\) denote the standard dyadic grid in \(\mathbb R^d\).
For \(\vec w\in\{0,\tfrac12\}^d\), define the family of translated cubes
\[
    \mathcal D^{\vec w}
    :=
    \left\{
      2^k\vec w+
      \prod_{j=1}^d[m_j2^k,(m_j+1)2^k):
      k\in\mathbb Z,\ (m_1,\ldots,m_d)\in\mathbb Z^d
    \right\}.
\]
The collection $\mathcal D^{\vec w}$ need not be nested when
$\vec w\ne0$. We use it as a family of translated cubes at each
scale. Its children, however, are standard dyadic cubes.
Indeed, indexing a child by $\varepsilon\in\{0,1\}^d$, we have
\[
    K^\varepsilon
    =2^{k-1}\bigl(2m+2\vec w+\varepsilon+[0,1)^d\bigr)
    \in\mathcal D,
\]
because $2m+2\vec w+\varepsilon\in\mathbb Z^d$.
In particular, for each fixed child index $\iota$,
any two intersecting members of
$\{K^\iota:K\in\mathcal D^{\vec w}\}$
are comparable by inclusion.
At every fixed scale $k\in\mathbb Z$, these $2^d$ translated families
satisfy the partition identity
\begin{equation}\label{eq:shifted-grid-partition}
    \sum_{\vec w\in\{0,\frac12\}^d}
    \ \sum_{\substack{K\in\mathcal D^{\vec w}\\ s_K=k}}
    \mathbf 1_{\frac12 K}=1
    \qquad\text{almost everywhere on }\mathbb R^d.
\end{equation}
This identity will be used in Section~\ref{sec:sparse-domination} to
localize the dyadic kernel pieces in physical space.

\subsection{Stopping collections and localized spaces}
\label{subsec:stopping-collections}

We recall the localized stopping-time framework of
Di Plinio, Hyt\"onen, and Li~\cite[Section~2]{MR4245601}. Let \(Q_0\) be a
dyadic cube. A \emph{stopping collection with top \(Q_0\)} is a family
\(\mathcal Q\) of dyadic cubes satisfying the following Whitney-type
conditions. Define
\[
    \operatorname{sh}\mathcal Q
    :=\bigcup_{L\in\mathcal Q}L,
    \qquad
    \mathrm c\mathcal Q
    :=\{L\in\mathcal Q:3L\cap 2Q_0\neq\emptyset\}.
\]
Then
\begin{equation}\label{eq:stopping-shadow}
    \bigcup_{L\in\mathrm c\mathcal Q}9L
    \subset \operatorname{sh}\mathcal Q
    \subset 3Q_0,
\end{equation}
\begin{equation}\label{eq:stopping-disjointness}
    L,L'\in\mathcal Q,\quad L\cap L'\neq\emptyset
    \quad\Longrightarrow\quad L=L',
\end{equation}
and
\begin{equation}\label{eq:stopping-neighbours}
    L'\in N(L)
    \quad\Longrightarrow\quad
    |s_L-s_{L'}|\leq 8,
    \qquad
    N(L):=\{L'\in\mathcal Q:3L\cap 3L'\neq\emptyset\}.
\end{equation}
In particular, \(\#N(L)\lesssim_d 1\) uniformly in \(L\in\mathcal Q\).
Moreover, by \eqref{eq:stopping-disjointness} and
\eqref{eq:stopping-shadow},
\begin{equation}\label{eq:stopping-packing}
    \sum_{L\in\mathcal Q}|L|
    =|\operatorname{sh}\mathcal Q|
    \leq 3^d|Q_0|.
\end{equation}

We next define the localized spaces associated with \(\mathcal Q\). For
\(1\leq p\leq\infty\), let \(\mathcal Y_p(\mathcal Q)\) be the space of measurable
functions supported in \(3Q_0\) for which
\begin{equation}\label{eq:Yp-norm}
    \|f\|_{\mathcal{Y}_p(\mathcal Q)}
    :=
    \begin{cases}
    \displaystyle
    \max\left\{
      \|f\mathbf 1_{\mathbb R^d\setminus\operatorname{sh}\mathcal Q}\|_\infty,
      \ \sup_{L\in\mathcal Q}\inf_{x\in\widehat L}M_pf(x)
    \right\},
       & 1\leq p<\infty,\\[8pt]
    \|f\|_\infty,
       & p=\infty.
    \end{cases}
\end{equation}
Here \(\widehat L:=2^5L\) is the non-dyadic \(2^5\)-fold dilation of
\(L\).

The atomic space \(\mathcal{X}_p(\mathcal Q)\) is the subspace of
\(\mathcal{Y}_p(\mathcal Q)\) consisting of functions that admit a decomposition
\[
    b=\sum_{L\in\mathcal Q}b_L,
    \qquad
    \operatorname{supp}b_L\subset L.
\]
We write \(b\in\dot {\mathcal{X}}_p(\mathcal Q)\) if, in addition,
\[ \int_L b_L(x)\,dx=0
    \qquad\text{for every }L\in\mathcal Q.
\]
Both \(\mathcal{X}_p(\mathcal Q)\) and \(\dot{\mathcal{X}}_p(\mathcal Q)\) are equipped with
the norm inherited from \(\mathcal{Y}_p(\mathcal Q)\). When the stopping collection
is clear from the context, we abbreviate these norms by
\(\|\cdot\|_{\mathcal{Y}_p}\), \(\|\cdot\|_{\mathcal{X}_p}\), and
\(\|\cdot\|_{\dot{\mathcal{X}}_p}\).

We will repeatedly use the following immediate consequence of
\eqref{eq:Yp-norm}: for every \(L\in\mathcal Q\),
\begin{equation}\label{eq:atomic-size}
    \|b_L\|_{L^p}
    \lesssim_d
    |L|^{1/p}\|b\|_{\mathcal{X}_p(\mathcal Q)},
    \qquad 1\leq p<\infty.
\end{equation}
Indeed, for each \(x\in\widehat L\), one may choose a cube containing
both \(x\) and \(L\) whose measure is comparable to \(|L|\), and then
apply the definition of \(M_p\). In particular,
\[
    \|b_L\|_{L^1}
    \lesssim_d |L|\,\|b\|_{\dot {\mathcal{X}}_1(\mathcal Q)}
    \qquad\text{for }b\in\dot {\mathcal{X}}_1(\mathcal Q).
\]

\subsection{The Rademacher--Menshov inequality}
\label{subsec:rademacher-menshov}
We use the following $L^2$ estimate;
see \cite[Theorem~10.6]{MR2403711}.

\begin{lem}[Rademacher--Menshov]\label{lem:rademacher-menshov}
Let \((X,\mu)\) be a measure space, and let
\(f_1,\ldots,f_N\in L^2(X)\). Suppose that
\[
    \left\|\sum_{j=1}^N\varepsilon_jf_j\right\|_{L^2(X)}
    \leq B
    \qquad
    \text{for every }(\varepsilon_1,\ldots,\varepsilon_N)
    \in\{-1,1\}^N.
\]
Then
\[
    \left\|
      \max_{0\leq M\leq N}
      \left|\sum_{j=1}^Mf_j\right|
    \right\|_{L^2(X)}
    \leq C B\log(2+N),
\]
where \(C\) is an absolute constant.
\end{lem}

\section{Sparse domination for bounded angular kernels}
\label{sec:sparse-domination}

In this section, we establish the sparse domination estimate for bounded
angular kernels. The proof has two main ingredients: the abstract sparse
domination principle of Di Plinio, Hyt\"onen, and Li
\cite{MR4245601}, and the physical-space linearization introduced by Lai
\cite{lai2025weak11estimatemaximal}. We first reduce the continuous maximal
truncation to a dyadic model and then place that model in the abstract
framework. 
The remainder of the section is devoted to the localized
estimates required by the sparse domination principle.
\subsection{Dyadic reduction and the abstract sparse domination principle}
\label{subsec:dyadic-abstract-reduction}

Throughout this section, let \(d\geq2\) and let
\(\Omega\in L^\infty(S^{d-1})\) satisfy the hypotheses of
Theorem~\ref{thm:main}. We now reduce the standard maximal truncation
\(T_\Omega^\ast\) to a dyadic model.
Choose a nonnegative radial function
\(\psi\in C_c^\infty(\mathbb R^d)\), supported in
\(\{2^{-4}<|x|<2^{-2}\}\), such that
\[
    \sum_{k\in\mathbb Z}\psi(2^{-k}x)=1,
    \qquad x\neq 0.
\]
For $k\in\mathbb Z$, set
\[
K_k^\Omega(z)
:=
\psi(2^{-k}z)\frac{\Omega(z/|z|)}{|z|^d},
\qquad
T_k^\Omega f:=K_k^\Omega*f.
\]
Since $\Omega$ is fixed throughout this section, we abbreviate
\[
K_k:=K_k^\Omega,
\qquad
T_k:=T_k^\Omega.
\]
We then define the dyadic maximal truncation by
\[
    T_{\Omega,\ast}f(x)
    :=
    \sup_{m<n}
    \left|
      \sum_{m<k\leq n}T_k f(x)
    \right|.
\]
The discrepancy between an arbitrary radial truncation and a dyadic one is
controlled by the rough maximal operator
\[
    M_\Omega f(x)
    :=
    \sup_{\rho>0}\frac{1}{\rho^d}
    \int_{|x-y|<\rho}
       \left|
         \Omega\bigl((x-y)/|x-y|\bigr)
       \right|
       |f(y)|\,dy.
\]
More precisely,
\[
\begin{aligned}
     T_{\Omega}^{\ast}f(x)
    \lesssim_d
    T_{\Omega,\ast}f(x)&+M_\Omega f(x),
    \qquad
    T_{\Omega,\ast}f(x)
    \lesssim_d
    T_{\Omega}^{\ast}f(x)+M_\Omega f(x),\\
   & M_\Omega f(x)
    \lesssim_d
    \|\Omega\|_{L^\infty(S^{d-1})}Mf(x).
\end{aligned}
\]
The first two comparisons are geometric and remain valid for every
\(\Omega\in L^1(S^{d-1})\); boundedness is used only in the third.
Since the Hardy--Littlewood maximal operator admits a sparse
\((1,1)\)-bound, it remains to estimate \(T_{\Omega,\ast}\).

We now recall the abstract kernel framework from \cite{MR4245601}. Let
\([K]=\{K_s:s\in\mathbb Z\}\) be a family of measurable kernels on
\(\mathbb R^d\times\mathbb R^d\) satisfying
\begin{equation}\label{equ3.1}
\begin{aligned}
    \operatorname{supp}K_s
    &\subset
    \bigl\{(x,y)\in\mathbb R^d\times\mathbb R^d:
           |x-y|<2^s\bigr\},\\
    \|[K]\|
    &:=
    \sup_{s\in\mathbb Z}2^{sd}
    \sup_{x\in\mathbb R^d}
    \left(
      \|K_s(x,\cdot)\|_\infty
      +
      \|K_s(\cdot,x)\|_\infty
    \right)
    <\infty.
\end{aligned}
\end{equation}
Associated with \([K]\) are the truncated operators
\[
    T[K]f(x,t_1,t_2)
    :=
    \sum_{t_1<s\leq t_2}
    \int_{\mathbb R^d}K_s(x,y)f(y)\,dy,
    \qquad
    x\in\mathbb R^d,\quad t_1,t_2\in\mathbb Z.
\]
All scale sums are understood to be empty when the upper
truncation parameter does not exceed the lower one.
The corresponding maximal truncations are
\[
\begin{aligned}
    T_{\star,t_1}^{t_2}[K]f(x)
    &:=
    \sup_{t_1\leq\tau_1\leq\tau_2\leq t_2}
    \bigl|T[K]f(x,\tau_1,\tau_2)\bigr|,\\
    T_\star[K]f(x)
    &:=
    \sup_{t_1\leq t_2}
    \bigl|T[K]f(x,t_1,t_2)\bigr|.
\end{aligned}
\]
We assume that, for some \(1<r<\infty\),
\begin{equation}\label{equ3.3}
    \|[K]\|_{r,\star}
    :=
    \|T_\star[K]\|_{L^r(\mathbb R^d)\to L^r(\mathbb R^d)}
    <\infty.
\end{equation}

For measurable integer-valued functions
\(t_1,t_2:\mathbb R^d\to\mathbb Z\) satisfying \(t_1\leq t_2\)
pointwise, define the linearized truncation
\[
    T[K]_{t_1}^{t_2}f(x)
    :=
    T[K]f\bigl(x,t_1(x),t_2(x)\bigr),
    \qquad x\in\mathbb R^d.
\]
Let \(\mathcal Q\) be a stopping collection with top \(Q_0\), and let
\(t_1,t_2\) be bounded measurable integer-valued functions. The associated
localized truncated form is
\begin{equation}\label{equ3.5}
\begin{aligned}
    \mathcal Q[K]_{t_1}^{t_2}(f_1,f_2)
    :=
    \frac{1}{|Q_0|}
    \Bigg(
      \left\langle
        T[K]_{t_1}^{\,t_2\wedge s_{Q_0}}
        (f_1\mathbf 1_{Q_0}),
        f_2
      \right\rangle
      -
      \sum_{\substack{L\in\mathcal Q\\L\subset Q_0}}
      \left\langle
        T[K]_{t_1}^{\,t_2\wedge s_L}
        (f_1\mathbf 1_L),
        f_2
      \right\rangle
    \Bigg),
\end{aligned}
\end{equation}
where \(t_2\wedge s_L\) denotes the pointwise minimum of \(t_2\) and
the constant function \(s_L\). The second term in \eqref{equ3.5} removes
the contributions localized below the stopping scales. Moreover, the
support condition in \eqref{equ3.1} implies that
\[
    \mathcal Q[K]_{t_1}^{t_2}(f_1,f_2)
    =
    \mathcal Q[K]_{t_1}^{t_2}
    (f_1\mathbf 1_{Q_0},f_2\mathbf 1_{3Q_0}).
\]

In the abstract framework, we use the same symbol for the associated
two-variable kernel,
\[
    K_s^\Omega(x,y):=K_s^\Omega(x-y).
\]
Thus its meaning is determined by its arguments or by the occurrence
of convolution. With this convention, set
\[
    [K^\Omega]:=\{K_s^\Omega:s\in\mathbb Z\}.
\]
Then \eqref{equ3.1} holds with
$
\|[K^\Omega]\|
\lesssim_d
\|\Omega\|_{L^\infty(S^{d-1})},
$
and
$
T_\star[K^\Omega]=T_{\Omega,\ast}.
$
The required
\(L^r\)-boundedness in \eqref{equ3.3} follows from the classical maximal
rough singular integral estimate; see \cite{MR837527}.

The following result is the abstract sparse domination principle on which
our proof is based.

\begin{thm}[{\cite[Theorem~3.3]{MR4245601}}]\label{thm3.6}
Let \([K]=\{K_s:s\in\mathbb Z\}\) be a family of kernels satisfying
\eqref{equ3.1} and \eqref{equ3.3}. Suppose that
\(1\leq p_1,p_2<\infty\) and
\begin{equation}\label{equ3.6}
\begin{aligned}
    C_L[K](p_1,p_2)
    :=\sup_{\mathcal Q,t_1,t_2}
    \Bigg\{
    \sup_{\substack{
        \|b\|_{\dot{\mathcal X}_{p_1}(\mathcal Q)}=1\\
        \|f\|_{\mathcal Y_{p_2}(\mathcal Q)}=1}}
      \bigl|
        \mathcal Q[K]_{t_1}^{t_2}(b,f)
      \bigr|
    +
    \sup_{\substack{
        \|f\|_{\mathcal Y_\infty(\mathcal Q)}=1\\
        \|b\|_{\dot{\mathcal X}_{p_2}(\mathcal Q)}=1}}
      \bigl|
        \mathcal Q[K]_{t_1}^{t_2}(f,b)
      \bigr|
    \Bigg\}
    <\infty,
\end{aligned}
\end{equation}
where the outer supremum is taken over all stopping collections
\(\mathcal Q\) and all bounded measurable integer-valued truncation
functions \(t_1,t_2\). Then
\begin{equation}\label{equ3.7}
    \|T_\star[K]\|_{(p_1,p_2)\text{-}\mathrm{sparse}}
    \lesssim
    \|[K]\|_{r,\star}+C_L[K](p_1,p_2).
\end{equation}
The implicit constant in \eqref{equ3.7} is independent of
\(p_1\) and \(p_2\). In our application \(r=2\) is fixed.
\end{thm}

We apply Theorem~\ref{thm3.6} with \(p_1=1\) and \(p_2=p\), where
\(1<p<\infty\). For the rough kernel family
\([K]:=\{K_k:k\in\mathbb Z\}\) defined above, we suppress the kernel
from the notation and write
\[
    \mathcal Q_{t_1}^{t_2}
    :=
    \mathcal Q[K]_{t_1}^{t_2}.
\]
According to \eqref{equ3.6}, it is enough to establish the two localized testing estimates
\[
\bigl|\mathcal Q_{t_1}^{t_2}(b,h)\bigr|
\lesssim_d p'
\|\Omega\|_{L^\infty(S^{d-1})}
\|b\|_{\dot{\mathcal X}_1(\mathcal Q)}
\|h\|_{\mathcal Y_p(\mathcal Q)},
\]
and
\[
\bigl|\mathcal Q_{t_1}^{t_2}(h,b)\bigr|
\lesssim_d p'
\|\Omega\|_{L^\infty(S^{d-1})}
\|h\|_{\mathcal Y_\infty(\mathcal Q)}
\|b\|_{\dot{\mathcal X}_p(\mathcal Q)}.
\]
The first estimate is the main work of this section. 
For the second estimate, we combine the quantitative weak
$L^{2p'}$ bound from \cite[(1.9)]{MR4245601} with a
finite-measure embedding on the localized supports.
The details are given in
Section~\ref{subsec:completion-bounded-kernel-case}.

\subsection{Expansion of the localized form}
\label{subsec:localized-form-expansion}
Fix a stopping collection $\mathcal Q$ with top $Q_0$.
We first assume that only finitely many atoms are nonzero and that
the truncation functions are bounded. All estimates below are uniform
in these finite restrictions. For bounded angular kernels, passage
to general atomic sums follows from $L^1$ convergence of the atoms:
there are only finitely many kernel scales in a fixed linearization,
each kernel is bounded, and $h\in L^1(3Q_0)$.
The corresponding passage for the unbounded angular part follows
from the single-atom estimate \eqref{eq:bad-single-atom}.
By the support
properties of the localized form, an atom \(b_L\) with
\(L\nsubseteq Q_0\) does not contribute to
\(\mathcal Q_{t_1}^{t_2}(b,h)\). We may therefore write
\[
    b
    =
    \sum_{\substack{L\in\mathcal Q\\L\subset Q_0}}b_L
    \in\dot{\mathcal X}_1(\mathcal Q),
    \qquad
    h\in\mathcal Y_p(\mathcal Q),
\]
where
\[
    \operatorname{supp}b_L\subset L,
    \qquad
    \int_L b_L(x)\,dx=0.
\]
The atomic size estimate \eqref{eq:atomic-size} and the disjointness of
the stopping cubes give
\[
    \|b_L\|_{L^1}
    \lesssim_d \|b\|_{\dot{\mathcal X}_1(\mathcal Q)}|L|,
    \qquad
    \sum_{\substack{L\in\mathcal Q\\L\subset Q_0}}
       \|b_L\|_{L^1}
    \lesssim_d \|b\|_{\dot{\mathcal X}_1(\mathcal Q)}|Q_0|.
\]

We next rewrite the localized form in terms of the separation between
the kernel scale and the stopping scale. By the definition
\eqref{equ3.5},
\[
\begin{aligned}
    \mathcal Q_{t_1}^{t_2}(b,h)
    =
    \frac{1}{|Q_0|}
    \sum_{\substack{L\in\mathcal Q\\L\subset Q_0}}
    \int_{\mathbb R^d}
       h(x)
       \sum_{s_L<k\leq s_{Q_0}}
       \mathbf 1_{\{t_1(x)<k\leq t_2(x)\}}
       T_kb_L(x)\,dx.
\end{aligned}
\]
For \(j\in\mathbb Z\), define
\[
    b_j
    :=
    \sum_{\substack{L\in\mathcal Q\\
                    L\subset Q_0\\
                    s_L=j}}
    b_L.
\]
If \(s:=k-s_L\), then the restriction \(s_L<k\) is exactly the
condition \(s\geq1\). Hence
\[
\begin{aligned}
    \mathcal Q_{t_1}^{t_2}(b,h)
    =
    \frac{1}{|Q_0|}
    \sum_{s\geq1}
    \int_{\mathbb R^d}
       h(x)
       \sum_{k\leq s_{Q_0}}
       \mathbf 1_{\{t_1(x)<k\leq t_2(x)\}}
       T_kb_{k-s}(x)\,dx.
\end{aligned}
\]
Thus \(s\) measures the gap between the physical scale \(2^k\) of the
kernel and the scale \(2^{k-s}\) of the cancellation atoms. This is the
localized counterpart of the scale decomposition in Lai's weak
\((1,1)\) argument. It is important to distinguish the two families of
cubes that occur below: the cubes \(L\in\mathcal Q\) support the
mean-zero atoms \(b_L\), whereas the cubes \(K\) are introduced only to
localize the operator in physical space.

We now apply the partition identity
\eqref{eq:shifted-grid-partition}.  For
\(K\in\mathcal D^{\vec w}\) with \(s_K=k\), define
\[
    T_Kg(x)
    :=
    \int_{\mathbb R^d}
       K_k(x-y)g(y)\mathbf 1_{\frac12K}(y)\,dy.
\]
Then
\[
    T_kg(x)
    =
    \sum_{\vec w\in\{0,\frac12\}^d}
    \sum_{\substack{K\in\mathcal D^{\vec w}\\s_K=k}}
    T_Kg(x).
\]

Only cubes \(K\) satisfying
\[
    b_{k-s}\mathbf 1_{\frac12K}\neq0
\]
can occur. Since \(b_{k-s}\) is supported in \(Q_0\), such a cube
satisfies
\(\frac12K\cap Q_0\neq\emptyset\). Moreover,
\(k\leq s_{Q_0}\), and hence \(\ell(K)\leq\ell(Q_0)\). A direct
geometric comparison therefore gives
\[
    K\subset 3Q_0.
\]
For notational compatibility with the later construction, we set
\[
    \widetilde Q_0^{\vec w}:=3Q_0.
\]
The superscript only records the translated cube family under consideration; in
particular,
\[
    |\widetilde Q_0^{\vec w}|
    =3^d|Q_0|.
\]

We record two elementary geometric facts about this physical
localization. First, fix \(s\geq1\), \(k\in\mathbb Z\), and a 
translated cube family \(\mathcal D^{\vec w}\). Let \(K,K'\in
\mathcal D^{\vec w}\) be distinct cubes with
\(\ell(K)=\ell(K')=2^k\). Since the cubes of scale \(2^k\) tile
\(\mathbb R^d\), their concentric half-cubes are separated in at least
one coordinate by a gap of length at least \(2^{k-1}\); equivalently,
\[
    \operatorname{dist}_\infty
    \left(\frac12K,\frac12K'\right)
    \geq 2^{k-1}.
\]
On the other hand, every stopping cube \(L\) contributing to
\(b_{k-s}\) has
\[
    \ell(L)=2^{k-s}\leq2^{k-1}.
\]
Consequently, \(L\) cannot have a positive-measure intersection with
both \(\frac12K\) and \(\frac12K'\). In particular, for fixed
\(s,k,\vec w\), each stopping atom \(b_L\) can contribute to at most
one of the operators \(T_K\). 

Second, the localization by \(\frac12K\) forces the output of \(T_K\)
to remain in \(K\). Indeed, if the integrand defining \(T_Kb_L(x)\) is
nonzero, then \(y\in\frac12K\) and, by the support of \(\psi\),
\[
    |x-y|<2^{k-2}.
\]
Every point \(y\in \frac12 K\) has \(\ell^\infty\)-distance at least
\(2^{k-2}\) from \(K^c\).
It follows, up to the
irrelevant boundaries of the cubes, that
\begin{equation}\label{equ3.12}
    \operatorname{supp}(T_Kb_L)\subset K.
\end{equation}

We finish this preliminary localization by proving the packing estimate
used below. This estimate is the localized analogue of
\cite[Lemma~3.3]{lai2025weak11estimatemaximal}, but that lemma cannot be
invoked verbatim here. In Lai's argument the scale separation is very
large, so the associated cancellation cube is contained in
\(\frac12K\). Here \(s\) is allowed to equal \(1\), and this containment
need not hold. The preceding separation argument provides the required
substitute.

Fix \(s\geq1\) and \(\vec w\). For each cube
\(K\in\mathcal D^{\vec w}\) such that, with \(k=s_K\),
\[
    K\subset\widetilde Q_0^{\vec w},
    \qquad
    b_{k-s}\mathbf 1_{\frac12K}\neq0,
\]
choose one stopping cube \(L_K\in\mathcal Q\) satisfying
\[
    L_K\subset Q_0,
    \qquad
    s_{L_K}=k-s,
    \qquad
    b_{L_K}\mathbf 1_{\frac12K}\neq0.
\]
Such a cube exists by the definition of \(b_{k-s}\). We claim that the
cubes \(L_K\) selected in this way are all distinct. 
Indeed, suppose that \(L_K=L_{K'}\). Since \(s\) is fixed,
we obtain \(s_K=s_{K'}\). Thus \(K\) and \(K'\) have the same scale.
Moreover, the common stopping cube \(L_K=L_{K'}\) contributes to both
\(\frac12K\) and \(\frac12K'\). The uniqueness established above
therefore forces \(K=K'\). Hence \(K\mapsto L_K\) is injective.

Since
\[
    |K|=2^{ds}|L_K|,
\]
and since the stopping cubes \(L_K\) are pairwise disjoint and contained
in \(Q_0\), we obtain
\[
\begin{aligned}
    \sum_{\substack{
        K\in\mathcal D^{\vec w}\\
        K\subset\widetilde Q_0^{\vec w}\\
        b_{s_K-s}\mathbf 1_{\frac12K}\neq0}}
        |K|
    =
    2^{ds}
    \sum_{\substack{
        K\in\mathcal D^{\vec w}\\
        K\subset\widetilde Q_0^{\vec w}\\
        b_{s_K-s}\mathbf 1_{\frac12K}\neq0}}
        |L_K|
    \leq
    2^{ds}
    \sum_{\substack{L\in\mathcal Q\\L\subset Q_0}}|L|
    \leq
    2^{ds}|Q_0|.
\end{aligned}
\]
This is the required packing bound for the physical localization cubes.

\begin{defn}[$K^{\iota}$]\label{def3.12}
For \(K\in\mathcal D^{\vec w}\), bisect each side of \(K\), and denote the
resulting \(2^d\) congruent subcubes by
\[
    K^1,K^2,\ldots,K^{2^d}.
\]
We fix the indexing once and for all according to the relative position of
the subcube inside \(K\). Thus
\[
    K=\bigcup_{\iota=1}^{2^d}K^\iota
\]
up to boundaries. For each fixed \(\iota\), the family
\[
    \{K^\iota:K\in\mathcal D^{\vec w}\}
\]
is nested: any two of its members are either disjoint or one contains the
other.
\end{defn}

For fixed \(s\geq1\) and
\(\vec w\in\{0,\frac12\}^d\), let
\[
    \mathscr K_s^{\vec w}
    :=
    \left\{
      K\in\mathcal D^{\vec w}:
      K\subseteq\widetilde Q_0^{\vec w},\
      s_K\leq s_{Q_0},\
      b_{s_K-s}\mathbf 1_{\frac12K}\not\equiv0
    \right\}.
\]
Thus \(\mathscr K_s^{\vec w}\) is precisely the collection of physical
localization cubes that contribute for the fixed parameters \(s\) and
\(\vec w\).

The packing argument above also has the following local form: for every
cube \(R\) in the nested family associated with the fixed value of
\(\iota\),
\begin{equation}\label{eq:local-packing-K-iota}
    \sum_{\substack{
        K\in\mathscr K_s^{\vec w}\\
        K^\iota\subseteq R}}
        |K^\iota|
    \lesssim_d
    2^{ds}|R|.
\end{equation}
To prove \eqref{eq:local-packing-K-iota}, use the
injective assignment $K\mapsto L_K$ constructed above.
For $K^\iota\subseteq R$, we have
\[
    L_K\cap\tfrac12K\neq\emptyset,
    \qquad
    \ell(L_K)\leq\tfrac12\ell(K),
\]
and hence $L_K\subseteq\tfrac32K$.
Moreover,
\[
    \ell(K)\leq2\ell(R),
    \qquad
    |c_K-c_R|_\infty\leq\tfrac12\ell(R).
\]
The latter inequality follows by comparing the centers
of $K^\iota$, $K$, and $R$. Consequently,
\[
    \sup_{x\in L_K}|x-c_R|_\infty
    \leq\tfrac34\ell(K)+\tfrac12\ell(R)
    \leq2\ell(R),
\]
so $L_K\subseteq4R$.
Since the selected stopping cubes are pairwise disjoint,
\[
\begin{aligned}
    \sum_{\substack{
        K\in\mathscr K_s^{\vec w}\\K^\iota\subseteq R}}
      |K^\iota|
    =
    2^{d(s-1)}
    \sum_{\substack{
        K\in\mathscr K_s^{\vec w}\\K^\iota\subseteq R}}
      |L_K|
    \leq2^{d(s-1)}|4R|
    \lesssim_d2^{ds}|R|.
\end{aligned}
\]

Combining the scale reindexing, the partition
\eqref{eq:shifted-grid-partition}, and the support
property \eqref{equ3.12}, and then decomposing each $K$
into its children, we obtain
\begin{equation}\label{equ3.13}
\begin{aligned}
  \mathcal Q_{t_1}^{t_2}(b,h)
  &=
  \frac1{|Q_0|}
  \sum_{s\geq1}
  \sum_{\vec w\in\{0,\frac12\}^d}
  \sum_{K\in\mathscr K_s^{\vec w}}
  \int_{\mathbb R^d}
      \mathbf 1_{\{t_1(x)<s_K\leq t_2(x)\}}
      T_Kb_{s_K-s}(x)h(x)\,dx
  \\
  &=
  \frac1{|Q_0|}
  \sum_{s\geq1}
  \sum_{\vec w\in\{0,\frac12\}^d}
  \sum_{\iota=1}^{2^d}
  \sum_{K\in\mathscr K_s^{\vec w}}
  \int_{\mathbb R^d}
      \mathbf 1_{\{t_1(x)<s_K\leq t_2(x)\}}
      T_Kb_{s_K-s}(x)
      \mathbf 1_{K^\iota}(x)h(x)\,dx .
\end{aligned}
\end{equation}

\subsection{Occupancy decomposition and the Rademacher--Menshov reduction}
\label{subsec:occupancy-rademacher-menshov}

We next organize the cubes according to their local overlap depth. This
filtration is adapted from the construction in
\cite{lai2025weak11estimatemaximal}, but we formulate it directly in terms
of the occupancy of the nested family \(\{K^\iota\}\). 
For the quantitative argument it is enough to consider
\(1<p\leq2\), so that \(p'\geq2\); the range \(p\geq2\) will be
recovered from the estimate at \(p=2\). Fix \(s\), \(\vec w\), and
\(\iota\), and suppress the dependence on \(\vec w\) in the notation.

Let \(A_d\geq1\) be a dimensional integer dominating the constants
in the global and local packing estimates used below. We choose
\[
    u_0:=4A_d\,2^{ds}.
\]
This choice gives a contraction factor at most \(1/4\) at each
occupancy generation and
\[
    \log(2+u_0)\lesssim_d 1+s.
\]
These bounds suffice for the quantitative argument in
Section~\ref{subsec:completion-bounded-kernel-case}.

\begin{defn}[$F_{s,\iota}^n$]\label{def:F-s-iota-n}
Set \(F_{s,\iota}^0:=\widetilde Q_0^{\vec w}\), and define
\[
    F_{s,\iota}^1
    :=
    \left\{
      x\in\mathbb R^d:
      \sum_{K\in\mathscr K_s^{\vec w}}
          \mathbf 1_{K^\iota}(x)>u_0
    \right\}.
\]
Recursively, for \(n\geq2\), let
\[
    F_{s,\iota}^n
    :=
    \left\{
      x\in\mathbb R^d:
      \sum_{\substack{K\in\mathscr K_s^{\vec w}\\
                      K^\iota\subseteq F_{s,\iota}^{n-1}}}
          \mathbf 1_{K^\iota}(x)>u_0
    \right\}.
\]
\end{defn}

For a fixed point \(x\), all cubes \(K^\iota\) containing \(x\) form a
chain. Thus \(F_{s,\iota}^1\) consists of the points at which the original
chain has length greater than \(u_0\), while
\(F_{s,\iota}^n\) records the points at which more than \(u_0\) members of
the chain remain after restricting to cubes contained in
\(F_{s,\iota}^{n-1}\). In particular,
\[
    F_{s,\iota}^0\supseteq F_{s,\iota}^1
    \supseteq F_{s,\iota}^2\supseteq\cdots.
\]

The following decay estimate is the localized counterpart of
\cite[Lemma~3.6]{lai2025weak11estimatemaximal} and follows from the
global packing bound above together with its local form
\eqref{eq:local-packing-K-iota}.
\begin{lem}\label{lem:F-s-iota-measure}
For every \(n\geq1\) and \(1\leq\iota\leq2^d\),
\begin{equation}\label{eq:F-s-iota-measure}
    |F_{s,\iota}^n|
    \lesssim_d
    2^{-2n}|Q_0|.
\end{equation}
\end{lem}

\begin{proof}
For \(n=1\), Chebyshev's inequality and the global packing estimate
give
\[
\begin{aligned}
    |F_{s,\iota}^1|
    \leq
    \frac1{u_0}
    \sum_{K\in\mathscr K_s^{\vec w}}|K^\iota|
    \leq
    \frac14|Q_0|
    \leq
    2^{-2}|Q_0|.
\end{aligned}
\]
For \(n\geq2\), decompose \(F_{s,\iota}^{n-1}\), up to a null set,
into its pairwise disjoint maximal components
\(\{R_\alpha\}_\alpha\) from the nested family \(\{K^\iota\}\).
Chebyshev's inequality and the local packing estimate yield
\[
\begin{aligned}
    |F_{s,\iota}^n|
    &\leq
    \frac1{u_0}
    \sum_\alpha
    \sum_{\substack{K\in\mathscr K_s^{\vec w}\\
                    K^\iota\subset R_\alpha}}
       |K^\iota|
    \leq
    \frac14
    \sum_\alpha|R_\alpha|
    \leq
    2^{-2}|F_{s,\iota}^{n-1}|.
\end{aligned}
\]
Iteration proves
\[
    |F_{s,\iota}^n|
    \lesssim_d
    2^{-2n}|Q_0|.
\]
\end{proof}

We use the first level at which a cube is no longer fully contained in the
occupancy set to partition \(\mathscr K_s^{\vec w}\).

\begin{defn}[$\mathcal I_{s,\iota}^{\sharp,n}$]
\label{def:I-s-iota-sharp}
For \(n=1\), set
\[
    \mathcal I_{s,\iota}^{\sharp,1}
    :=
    \left\{
      K\in\mathscr K_s^{\vec w}:
      K^\iota\nsubseteq F_{s,\iota}^1
    \right\}.
\]
For \(n\geq2\), set
\[
    \mathcal I_{s,\iota}^{\sharp,n}
    :=
    \left\{
      K\in\mathscr K_s^{\vec w}:
      K^\iota\subseteq F_{s,\iota}^{n-1}
      \ \text{and}\
      K^\iota\nsubseteq F_{s,\iota}^{n}
    \right\}.
\]
\end{defn}

Since \(|F_{s,\iota}^n|\to0\), no cube of positive measure can be
contained in every \(F_{s,\iota}^n\). Hence each
\(K\in\mathscr K_s^{\vec w}\) belongs to exactly one
\(\mathcal I_{s,\iota}^{\sharp,n}\). Using \eqref{equ3.12} once more, we
therefore obtain
\[
\begin{aligned}
    \sum_{K\in\mathscr K_s^{\vec w}}T_Kb_{s_K-s}(x)
    &=
    \sum_{\iota=1}^{2^d}
    \sum_{K\in\mathscr K_s^{\vec w}}
        T_Kb_{s_K-s}(x)\mathbf 1_{K^\iota}(x)
    \\
    &=
    \sum_{\iota=1}^{2^d}
    \sum_{n\geq1}
    \sum_{K\in\mathcal I_{s,\iota}^{\sharp,n}}
        T_Kb_{s_K-s}(x)\mathbf 1_{K^\iota}(x)
\end{aligned}
\]
for almost every \(x\). Consequently, \eqref{equ3.13} becomes
\[
\begin{aligned}
  \mathcal Q_{t_1}^{t_2}(b,h)
  &=
  \frac1{|Q_0|}
  \sum_{s\geq1}
  \sum_{\vec w\in\{0,\frac12\}^d}
  \sum_{\iota=1}^{2^d}
  \sum_{n\geq1}
  \int_{\mathbb R^d} h(x)\\
  &\qquad\times
  \sum_{K\in\mathcal I_{s,\iota}^{\sharp,n}}
      \mathbf1_{\{t_1(x)<s_K\leq t_2(x)\}}
      T_Kb_{s_K-s}(x)
      \mathbf1_{K^\iota}(x)\,dx .
\end{aligned}
\]
It remains to separate the geometric organization of the cubes from the
\(x\)-dependent truncation. Fix \(s\), \(\vec w\), \(\iota\), and \(n\),
and write
\[
    \mathscr L_{s,\iota}^{\sharp,n}
    :=
    \{K^\iota:K\in\mathcal I_{s,\iota}^{\sharp,n}\}.
\]
For every \(K^\iota\in\mathscr L_{s,\iota}^{\sharp,n}\),
there are at most \(u_0\) members of
\(\mathscr L_{s,\iota}^{\sharp,n}\) containing \(K^\iota\),
including \(K^\iota\) itself. 
To see this, take \(K\in\mathcal I_{s,\iota}^{\sharp,n}\) and
choose
\[
    x_K\in K^\iota\setminus F_{s,\iota}^n.
\]
Every such cube \(J^\iota\) contains \(x_K\) and is contained in
\(F_{s,\iota}^{n-1}\). If there were more than \(u_0\) such ancestors,
the definition of \(F_{s,\iota}^n\) would imply
\(x_K\in F_{s,\iota}^n\), a contradiction.

We now peel the family from the top. Let
\[
    \mathscr L_1^n
    :=
    \max\mathscr L_{s,\iota}^{\sharp,n},
\]
where \(\max\) denotes the collection of maximal cubes, and define
recursively
\[
    \mathscr L_{u+1}^n
    :=
    \max\left(
      \mathscr L_{s,\iota}^{\sharp,n}
      \setminus\bigcup_{r=1}^{u}\mathscr L_r^n
    \right).
\]
Each \(\mathscr L_u^n\) is pairwise disjoint. The ancestor bound shows that
\(\mathscr L_u^n=\emptyset\) for \(u>u_0\). For
\(1\leq u\leq u_0\), set
\[
    \mathcal M_u^n
    :=
    \{K\in\mathcal I_{s,\iota}^{\sharp,n}:K^\iota\in\mathscr L_u^n\},
    \qquad
    \beta_u(x)
    :=
    \sum_{K\in\mathcal M_u^n}
       T_Kb_{s_K-s}(x)\mathbf 1_{K^\iota}(x).
\]
The functions \(\beta_u\) depend on the fixed parameters
\((s,\vec w,\iota,n)\), but they do not depend on \(t_1\) or \(t_2\).

We finally explain rigorously how the \(x\)-dependence of the truncation
indices is transferred to the endpoints of the sequence
\(\{\beta_u\}_{u=1}^{u_0}\). Because the cubes in each
\(\mathscr L_u^n\) are disjoint, for fixed \(x\) there is at most one
\(K\in\mathcal M_u^n\) such that \(x\in K^\iota\). Moreover, the active
layers form an initial segment \(1,\ldots,m(x)\). Indeed, if a cube
containing \(x\) survives until the \(u\)-th peeling step, then at each
preceding step a maximal remaining ancestor of that cube is removed; all
these ancestors also contain \(x\). The corresponding cubes therefore
form a strictly decreasing chain
\[
    K_1(x)^\iota\supsetneq K_2(x)^\iota
    \supsetneq\cdots\supsetneq K_{m(x)}(x)^\iota.
\]
Consequently,
\[
    s_{K_1(x)}>s_{K_2(x)}>\cdots>s_{K_{m(x)}}.
\]
To include the harmless case \(t_1(x)\geq t_2(x)\), in which the
truncated sum is empty, define the integer-valued functions
\[
\begin{aligned}
    v_1(x)
    &:=
    \mathbf 1_{\{t_1(x)<t_2(x)\}}
    \#\{1\leq u\leq m(x):s_{K_u(x)}>t_1(x)\},\\
    v_2(x)
    &:=
    \mathbf 1_{\{t_1(x)<t_2(x)\}}
    \#\{1\leq u\leq m(x):s_{K_u(x)}>t_2(x)\}.
\end{aligned}
\]
They are measurable because, for \(j=1,2\),
\[
    v_j(x)
    =
    \mathbf 1_{\{t_1(x)<t_2(x)\}}
    \sum_{u=1}^{u_0}\sum_{K\in\mathcal M_u^n}
       \mathbf 1_{K^\iota}(x)
       \mathbf 1_{\{s_K>t_j(x)\}}.
\]
By construction,
\[
    0\leq v_2(x)\leq v_1(x)\leq u_0.
\]
The strict decrease of the scales implies that the lower cutoff
\(s_K>t_1(x)\) selects the first \(v_1(x)\) active layers, whereas
\(s_K>t_2(x)\) selects the first \(v_2(x)\) active layers. Therefore,
with
\[
    S_v(x):=\sum_{u=1}^{v}\beta_u(x),
    \qquad S_0(x):=0,
\]
we have, on every active layer,
\[
    \mathbf 1_{\{t_1(x)<s_{K_u(x)}\leq t_2(x)\}}
    =
    \mathbf 1_{\{u\leq v_1(x)\}}
    -
    \mathbf 1_{\{u\leq v_2(x)\}}.
\]
Both sides vanish when \(t_1(x)\geq t_2(x)\). It follows that
\[
\begin{aligned}
    \sum_{K\in\mathcal I_{s,\iota}^{\sharp,n}}
    \mathbf 1_{\{t_1(x)<s_K\leq t_2(x)\}}
    T_Kb_{s_K-s}(x)\mathbf 1_{K^\iota}(x)
    =S_{v_1(x)}(x)-S_{v_2(x)}(x).
\end{aligned}
\]
In particular,
\[
\begin{aligned}
&\left|
\sum_{K\in\mathcal I_{s,\iota}^{\sharp,n}}
    \mathbf 1_{\{t_1(x)<s_K\leq t_2(x)\}}
    T_Kb_{s_K-s}(x)\mathbf 1_{K^\iota}(x)
\right|
    \leq
    2\sup_{0\leq v\leq u_0}
       \left|\sum_{u=1}^{v}\beta_u(x)\right|.
\end{aligned}
\]

All dependence on $t_1,t_2$ is now confined to the two prefix
endpoints, while the functions $\beta_u$ are fixed.
In Section~\ref{subsec:completion-bounded-kernel-case}, we first apply
Lemma~\ref{lem:rademacher-menshov} at exponent $2$.
Thus a uniform signed bound $\|\sum_u\varepsilon_u\beta_u\|_2\leq B_2$
gives a maximal bound $C\log(2+u_0)B_2$ in $L^2$.
We then combine this maximal $L^2$ bound with a
positive-kernel $\mathcal Y_1$ pairing estimate.
A decomposition of the second input by amplitude gives
the required $\mathcal Y_p$ testing bound.

\subsection{Signed sums and microlocal endpoint estimates}
\label{subsec:signed-sums-microlocal-endpoints}
It remains to estimate the signed sums appearing in the preceding
Rademacher--Menshov reduction.

Before expanding the signed sum, we clarify the cube notation. The
physical localization cubes were denoted by \(K\) above. From this point
on, we denote the same cubes by \(J\), reserving \(P\) for the parent
cubes of maximal components introduced below. We also write $Q$, rather than $L$,
for an individual stopping cube in $\mathcal{Q}$. Thus \(J^\iota\) is the
\(\iota\)-th child of the physical cube \(J\); this is only a relabeling,
not a new family of cubes. For
\(J\in\mathcal I_{s,\iota}^{\sharp,n}\), let \(u(J)\) be the unique
integer such that \(J\in\mathcal M_{u(J)}^n\), and put
\(\varepsilon_J:=\varepsilon_{u(J)}\).

By the definition of \(\beta_u\),
\[
\begin{aligned}
    \sum_{u=1}^{u_0}\varepsilon_u\beta_u(x)
    &=
    \sum_{j\in\mathbb Z}
    \sum_{\substack{
        J\in\mathcal I_{s,\iota}^{\sharp,n}\\
        s_J=j}}
       \varepsilon_J
       T_Jb_{j-s}(x)\mathbf 1_{J^\iota}(x)
    \\
    &=
    \sum_{j\in\mathbb Z}
    \sum_{\substack{
        J\in\mathcal I_{s,\iota}^{\sharp,n}\\
        s_J=j}}
    \sum_{\substack{Q\in\mathcal Q\\s_Q=j-s}}
       \varepsilon_J
       T_j\!\left(
          b_Q\mathbf 1_{\frac12J}
       \right)(x)
       \mathbf 1_{J^\iota}(x).
\end{aligned}
\]
The second equality follows from
\[
    b_{j-s}
    =
    \sum_{\substack{Q\in\mathcal Q\\s_Q=j-s}}b_Q
\]
and from the definition
\[
    T_Jg=T_j\!\left(g\mathbf 1_{\frac12J}\right)
    \qquad (s_J=j).
\]
Notice that the cutoff \(\mathbf 1_{\frac12J}\) is still present. Whether
it can be removed depends on the scale separation \(s\); this is why the
cases \(s\geq2\) and \(s=1\) will be treated separately below. It would
also be sufficient to prove the signed-sum estimate for arbitrary signs
\(\varepsilon_J\), which is slightly stronger than the layerwise choice
\(\varepsilon_J=\varepsilon_{u(J)}\).

We introduce the collections of stopping cubes that will occur in these
two cases. For \(s\geq2\), define
\[
\begin{aligned}
    \mathfrak V_{j-s}^{n,\iota}
    &:=
    \left\{
      Q\in\mathcal Q:
      \begin{array}{l}
        s_Q=j-s,\ \text{and there exists }
        J\in\mathcal I_{s,\iota}^{\sharp,n}\\
        \text{with }s_J=j\text{ and }Q\subseteq\frac12J
      \end{array}
    \right\},
    \\
    \mathfrak V^{n,\iota}
    &:=
    \bigcup_{j\in\mathbb Z}\mathfrak V_{j-s}^{n,\iota}.
\end{aligned}
\]
For fixed \(j\) and \(\vec w\), the cubes \(\frac12J\) are pairwise
disjoint. Hence the cube \(J\) associated with
\(Q\in\mathfrak V_{j-s}^{n,\iota}\) is unique. For \(s=1\), it is more
convenient to retain the pair \((Q,J)\) and set
\[
    \widetilde{\mathfrak V}_{j-1}^{n,\iota}
    :=
    \left\{
      (Q,J):
      \begin{array}{l}
        Q\in\mathcal Q,\ s_Q=j-1,\quad
        J\in\mathcal I_{1,\iota}^{\sharp,n},\ s_J=j,\\
        b_Q\mathbf 1_{\frac12J}\not\equiv0
      \end{array}
    \right\}.
\]
Although the preceding separation argument gives uniqueness of \(J\)
for fixed \(Q\), \(j\), and \(\vec w\), up to null boundaries, we retain
the pair \((Q,J)\) in the notation because the truncated atom
\[
    b_{Q,J}:=b_Q\mathbf 1_{\frac12J}
\]
depends explicitly on the physical localization cube \(J\).

We next establish the packing estimate used in both cases. Fix
\(n\geq2\). Let
\[
    \mathscr C_{s,\iota}^{n-1}
    :=
    \max\left\{
       P^\iota:
       P\in\mathscr K_s^{\vec w},\
       P^\iota\subseteq F_{s,\iota}^{n-1}
    \right\},
\]
where \(\max\) denotes the maximal members with respect to inclusion.
These cubes are pairwise disjoint and, up to their boundaries,
\[
    F_{s,\iota}^{n-1}
    =
    \bigcup_{P^\iota\in\mathscr C_{s,\iota}^{n-1}}P^\iota.
\]
For each component \(P^\iota\), the symbol \(P\) denotes the unique
physical cube whose \(\iota\)-th child is \(P^\iota\). Define the union of
these parent cubes by
\[
    \widetilde F_{s,\iota}^{\,n-1}
    :=
    \bigcup_{P^\iota\in\mathscr C_{s,\iota}^{n-1}}P.
\]
Since \(|P|=2^d|P^\iota|\), the disjointness of the maximal
components and Lemma~\ref{lem:F-s-iota-measure} give, for \(n\geq2\),
\begin{equation}\label{eq:F-tilde-measure}
\begin{aligned}
    |\widetilde F_{s,\iota}^{\,n-1}|
    \leq
    \sum_{P^\iota\in\mathscr C_{s,\iota}^{n-1}}|P|
    =
    2^d
    \sum_{P^\iota\in\mathscr C_{s,\iota}^{n-1}}|P^\iota|
    =
    2^d|F_{s,\iota}^{n-1}|
    \lesssim_d
    2^{-2(n-1)}|Q_0|.
\end{aligned}
\end{equation}

For \(n=1\), there is no preceding exceptional set to decompose. We
set
\[
    \widetilde F_{s,\iota}^{\,0}
    :=
    \widetilde Q_0^{\vec w}.
\]
By the definition of the top localization region,
\begin{equation}\label{eq:F-tilde-measure-generation-zero}
    |\widetilde F_{s,\iota}^{\,0}|
    =3^d|Q_0|
    \lesssim_d|Q_0|.
\end{equation}

We claim that, for every \(n\geq1\),
\begin{equation}\label{eq:I-contained-F-tilde}
    \bigcup_{J\in\mathcal I_{s,\iota}^{\sharp,n}}J
    \subseteq
    \widetilde F_{s,\iota}^{\,n-1}.
\end{equation}
When \(n=1\), this follows directly from the restriction
\(J\subseteq\widetilde Q_0^{\vec w}\) in the definition of
\(\mathscr K_s^{\vec w}\). Suppose that \(n\geq2\). Then
\[
    J^\iota\subseteq F_{s,\iota}^{n-1}
    \qquad
    \text{for every }
    J\in\mathcal I_{s,\iota}^{\sharp,n}.
\]
Hence \(J^\iota\) is contained in a unique maximal component
\(P^\iota\in\mathscr C_{s,\iota}^{n-1}\).
Since \(J^\iota\subseteq P^\iota\) and the two children occupy
the same relative position in their respective parent cubes,
comparison of the coordinate endpoints gives \(J\subseteq P\).
Hence
\(J\subseteq P\subseteq\widetilde F_{s,\iota}^{\,n-1}\),
which proves \eqref{eq:I-contained-F-tilde}.

We can now prove the desired packing bound. By
\eqref{eq:atomic-size},
\[
    \|b_Q\|_{L^1}
    \lesssim_d
    \|b\|_{\dot{\mathcal X}_1(\mathcal Q)}|Q|
    \qquad (Q\in\mathcal Q).
\]
The stopping cubes are pairwise disjoint by
\eqref{eq:stopping-disjointness}, and every
\(Q\in\mathfrak V^{n,\iota}\) is contained in its associated physical
cube \(J\). It follows from \eqref{eq:I-contained-F-tilde} and
\eqref{eq:F-tilde-measure}, with
\eqref{eq:F-tilde-measure-generation-zero} used when \(n=1\), that
\begin{equation}\label{eq:packing-bQ-revised}
\begin{aligned}
    \sum_{Q\in\mathfrak V^{n,\iota}}\|b_Q\|_{L^1}
    \lesssim_d
    \|b\|_{\dot{\mathcal X}_1(\mathcal Q)}
    \sum_{Q\in\mathfrak V^{n,\iota}}|Q|
    \leq
    \|b\|_{\dot{\mathcal X}_1(\mathcal Q)}
    \left|
       \bigcup_{J\in\mathcal I_{s,\iota}^{\sharp,n}}J
    \right|
    \lesssim_d
    2^{-2(n-1)}
    \|b\|_{\dot{\mathcal X}_1(\mathcal Q)}|Q_0|.
\end{aligned}
\end{equation}
The same packing argument applies when $s=1$:
the associated stopping cubes are distinct and lie in
their corresponding physical cubes.

We now separate the cases \(s\geq2\) and \(s=1\). The distinction is
caused by the physical cutoff \(\mathbf 1_{\frac12J}\) in the definition
of \(T_J\). For
\(Q\in\mathfrak V_{j-s}^{n,\iota}\), let \(J(Q)\) denote the unique
physical localization cube associated with \(Q\).

Suppose first that \(s\geq2\). Let \(J\in\mathcal D^{\vec w}\) satisfy
\(s_J=j\), and let \(Q\in\mathcal Q\) satisfy \(s_Q=j-s\). The boundary
of the concentric half-cube \(\frac12J\) lies on the lattice of dyadic
cubes of scale \(2^{j-s}\) whenever \(s\geq2\). Hence
\(\frac12J\) is a union of cubes of the same scale as \(Q\). It follows
that
\[
    Q\cap\frac12J\neq\emptyset
    \quad\Longrightarrow\quad
    Q\subseteq\frac12J
\]
up to boundaries. In particular, whenever \(Q\) contributes to
\(T_Jb_{j-s}\),
\[
    b_Q\mathbf 1_{\frac12J}=b_Q,
    \qquad
    \int_Qb_Q(x)\,dx=0,
\]
and therefore
\begin{equation}\label{eq:TJ-full-atom}
    T_Jb_Q
    =
    T_j\!\left(b_Q\mathbf 1_{\frac12J}\right)
    =
    T_jb_Q.
\end{equation}
Thus, for \(s\geq2\), the physical localization does not alter the
stopping atom, and its mean-zero cancellation remains available.

The situation changes when \(s=1\). Now
\(\ell(Q)=2^{j-1}=\frac12\ell(J)\), but the boundary of
\(\frac12J\) need not lie on the lattice of the stopping cubes at scale
\(2^{j-1}\). Consequently, a stopping cube \(Q\) may have a
positive-measure intersection with \(\frac12J\) without being contained
in \(\frac12J\). The operator \(T_J\) only sees the part of \(b_Q\)
inside this half-cube. We must therefore define
\[
    b_{Q,J}
    :=
    b_Q\mathbf 1_{\frac12J},
    \qquad
    (Q,J)\in\widetilde{\mathfrak V}_{j-1}^{n,\iota}.
\]
With this notation, the exact localization identity is
\begin{equation}\label{eq:TJ-truncated-atom-s1}
    T_Jb_Q
    =
    T_j\!\left(b_Q\mathbf 1_{\frac12J}\right)
    =
    T_jb_{Q,J}.
\end{equation}

Although $b_{Q,J}$ need not have mean zero, it satisfies
\[
    \operatorname{supp}b_{Q,J}\subseteq Q,
    \qquad
    \|b_{Q,J}\|_1\leq\|b_Q\|_1.
\]
These properties suffice for the positive-kernel
estimate used when $s=1$.

We record the corresponding packing estimate. For fixed \(j\) and \(\vec w\), 
the preceding separation argument shows
that each stopping cube \(Q\) occurs in at most one pair \((Q,J)\), up
to null boundaries. Moreover, the same dyadic geometry shows that every
such \(Q\) is contained in \(J\), although it need not be contained in
\(\frac12J\). Therefore \eqref{eq:packing-bQ-revised} gives
\begin{equation}\label{eq:packing-truncated-atoms-s1}
\begin{aligned}
    \sum_{j\in\mathbb Z}
    \sum_{(Q,J)\in\widetilde{\mathfrak V}_{j-1}^{n,\iota}}
       \|b_{Q,J}\|_{L^1}
    &\leq
    \sum_{j\in\mathbb Z}
    \sum_{(Q,J)\in\widetilde{\mathfrak V}_{j-1}^{n,\iota}}
       \|b_Q\|_{L^1}
    &\lesssim_d
    2^{-2(n-1)}
    \|b\|_{\dot{\mathcal X}_1(\mathcal Q)}|Q_0|.
\end{aligned}
\end{equation}

Fix \(s\geq2\). We first replace the sharp spatial cutoff by a smooth
one. This step isolates the contribution of a thin boundary layer of
\(J^\iota\) and allows us to apply frequency projections to the
remaining part. Let \(\varpi\in C_c^\infty(\mathbb R^d)\) be
nonnegative and radial, with
\[
 \operatorname{supp}\varpi\subseteq
 \bigl\{x\in\mathbb R^d:|x|\leq2^{-4}\bigr\},
 \qquad
 \int_{\mathbb R^d}\varpi(x)\,dx=1.
\]
For \(t\in\mathbb R\), set
\[
 \varpi_t(x):=2^{-td}\varpi(2^{-t}x)
 \quad\text{and}\quad
 P_tg(x):=(\varpi_t\ast g)(x).
\]
For compatibility with the notation used in the microlocal argument,
we write
\[
    \chi_E:=\mathbf 1_E
\]
throughout this subsection. Define
\[
 \widetilde{\chi}_{J^\iota}
 :=P_{j-s\kappa}\chi_{J^\iota},
\]
where \(0<\kappa<1\) will be fixed later. Then
\(0\leq\widetilde{\chi}_{J^\iota}\leq1\), and, for every
multi-index \(\alpha\),
\begin{equation}\label{4.5}
 \bigl|\partial_x^\alpha
 \widetilde{\chi}_{J^\iota}(x)\bigr|
 \lesssim_{\alpha,\varpi}
 2^{-(j-s\kappa)|\alpha|}.
\end{equation}
Indeed, this follows by differentiating the convolution and using
\[
 \|\partial^\alpha\varpi_{j-s\kappa}\|_{L^1}
 \lesssim_{\alpha,\varpi}
 2^{-(j-s\kappa)|\alpha|}.
\]
Moreover,
\(\chi_{J^\iota}-\widetilde{\chi}_{J^\iota}\) is supported in the
\(2^{j-s\kappa-4}\)-neighborhood of
\(\partial J^\iota\). Thus the first error created by this
regularization is confined to a thin boundary layer.

We next introduce the angular and frequency decompositions used for
the rough kernel. We employ the microlocal scheme originating in
\cite{MR1317232}, in the form adapted to maximal rough singular
integrals in \cite{lai2025weak11estimatemaximal}. Choose a parameter
\(\gamma\) with
$
 0<\gamma<\kappa<1.
$
For the fixed \(s\), let
\(\Theta_s=\{e_v^s\}_v\subset S^{d-1}\) be a maximal collection
satisfying
\[
 |e_v^s-e_{v'}^s|\geq2^{-s\gamma-4}
 \qquad (v\neq v').
\]
Maximality immediately gives the corresponding covering property:
for every \(\theta\in S^{d-1}\), there is an \(e_v^s\in\Theta_s\)
such that
$
 |\theta-e_v^s|\leq2^{-s\gamma-4}.
$
The separation of the points and a standard comparison of spherical
cap measures imply
$
 \#\Theta_s\lesssim_d2^{s\gamma(d-1)},
$
and the enlarged caps have overlap bounded by a constant depending
only on \(d\).

By assigning each point of the sphere to one of its nearest elements
of \(\Theta_s\), with an arbitrary measurable rule for resolving
ties, we obtain a measurable partition
\(\{E_v^s\}_v\) of \(S^{d-1}\) such that
\[
 e_v^s\in E_v^s,\qquad
 \operatorname{diam}(E_v^s)\leq2^{-s\gamma-2},
 \qquad
 S^{d-1}=\bigsqcup_v E_v^s.
\]
The small parameter \(\gamma\), which is not present in the original
form of Seeger's decomposition, leaves room for the interpolation
argument below. Its precise size will be chosen only after all losses
in \(s\) have been identified.

Recall that \(K_j\) denotes the kernel of \(T_j\). For \(x\neq0\) and fixed $s$, set
\[
 K_j^v(x):=
 K_j(x)\chi_{E_v^s}\left(\frac{x}{|x|}\right),
\]
and define
\[
 T_j^vg(x):=
 \int_{\mathbb R^d}K_j^v(x-y)g(y)\,dy.
\]
Because the sets \(E_v^s\) form a partition of the sphere, this gives
the exact angular decomposition
\[
 T_j=\sum_vT_j^v.
\]

To each angular sector we associate a projection onto frequencies
that are nearly orthogonal to its central direction. Let
\(\phi\in C_c^\infty(\mathbb R)\) be nonnegative and even, with
\[
 0\leq\phi\leq1,\qquad
 \phi(t)=1\ \text{for }|t|\leq2,\qquad
 \phi(t)=0\ \text{for }|t|>4.
\]
For \(\xi\neq0\), define the Fourier multiplier \(G_v^s\) by
\[
 \widehat{G_v^sg}(\xi)
 :=
 \phi\left(
 2^{s\gamma}
 \left\langle e_v^s,\frac{\xi}{|\xi|}\right\rangle
 \right)\widehat g(\xi).
\]
The value of the multiplier at \(\xi=0\) may be assigned arbitrarily.
Thus \(G_v^s\) retains the frequency directions satisfying
\[
 \left|
 \left\langle e_v^s,\frac{\xi}{|\xi|}\right\rangle
 \right|
 \lesssim2^{-s\gamma},
\]
whereas \(I-G_v^s\) selects the complementary directions.

We now combine the spatial and microlocal decompositions. Fix
\(Q\in\mathfrak V_{j-s}^{n,\iota}\), and write \(J=J(Q)\). First
insert
\(\chi_{J^\iota}
=(\chi_{J^\iota}-\widetilde{\chi}_{J^\iota})
+\widetilde{\chi}_{J^\iota}\). On the second summand insert
\(I=P_{j-s\kappa}+(I-P_{j-s\kappa})\). Finally, use
\(T_j=\sum_vT_j^v\) in the high-frequency part and insert
\(I=G_v^s+(I-G_v^s)\) for each \(v\). These three exact identities
give
\[
\begin{aligned}
 \varepsilon_J(T_jb_Q)\chi_{J^\iota}
 ={}&
 \varepsilon_J(T_jb_Q)
 \bigl(\chi_{J^\iota}-\widetilde{\chi}_{J^\iota}\bigr)
 \\
 &+
 P_{j-s\kappa}
 \bigl[\varepsilon_J(T_jb_Q)
       \widetilde{\chi}_{J^\iota}\bigr]
 \\
 &+
 \sum_v
 (I-P_{j-s\kappa})G_v^s
 \bigl[\varepsilon_J(T_j^vb_Q)
       \widetilde{\chi}_{J^\iota}\bigr]
 \\
 &+
 \sum_v
 (I-P_{j-s\kappa})(I-G_v^s)
 \bigl[\varepsilon_J(T_j^vb_Q)
       \widetilde{\chi}_{J^\iota}\bigr].
\end{aligned}
\]
The four terms represent, respectively, the boundary error produced
by smoothing, the spatially smoothed low-frequency contribution, the
high-frequency contribution whose frequency direction is nearly
orthogonal to \(e_v^s\), and its complementary microlocal
contribution. The following endpoint estimates, together with the
positive-kernel $L^3$ bounds established below, yield the
uniform signed $L^2$ estimate used in
Section~\ref{subsec:completion-bounded-kernel-case}.

We next establish the endpoint estimates for the four
microlocal components. Fix \(s\geq2\), \(n\geq1\), and
\(1\leq\iota\leq2^d\).
After summing the four terms in the preceding decomposition over
\(j\) and \(Q\), denote them by
\[
\begin{aligned}
 \mathcal I_1
 &:=
 \sum_{j\in\mathbb Z}
 \sum_{Q\in\mathfrak V_{j-s}^{n,\iota}}
 \varepsilon_{J(Q)}(T_jb_Q)
 \bigl(
   \chi_{J(Q)^\iota}
   -\widetilde\chi_{J(Q)^\iota}
 \bigr),                                                    \\
 \mathcal I_2
 &:=
 \sum_{j\in\mathbb Z}
 \sum_{Q\in\mathfrak V_{j-s}^{n,\iota}}
 P_{j-s\kappa}
 \bigl[
   \varepsilon_{J(Q)}(T_jb_Q)
   \widetilde\chi_{J(Q)^\iota}
 \bigr],                                                    \\
 \mathcal I_3
 &:=
 \sum_{j\in\mathbb Z}
 \sum_{Q\in\mathfrak V_{j-s}^{n,\iota}}
 \sum_v
 (I-P_{j-s\kappa})G_v^s
 \bigl[
   \varepsilon_{J(Q)}(T_j^vb_Q)
   \widetilde\chi_{J(Q)^\iota}
 \bigr],                                                    \\
 \mathcal I_4
 &:=
 \sum_{j\in\mathbb Z}
 \sum_{Q\in\mathfrak V_{j-s}^{n,\iota}}
 \sum_v
 (I-P_{j-s\kappa})(I-G_v^s)
 \bigl[
   \varepsilon_{J(Q)}(T_j^vb_Q)
   \widetilde\chi_{J(Q)^\iota}
 \bigr].
\end{aligned}
\]
Thus
\[
    \sum_j\sum_{Q\in\mathfrak V_{j-s}^{n,\iota}}
       \varepsilon_{J(Q)}(T_jb_Q)\mathbf1_{J(Q)^\iota}
    =\mathcal I_1+\mathcal I_2+\mathcal I_3+\mathcal I_4.
\]

Before stating the endpoint estimates, we record the correspondence
with the microlocal argument of
\cite{lai2025weak11estimatemaximal}. For every
\(Q\in\mathfrak V_{j-s}^{n,\iota}\),
\[
\begin{gathered}
    \ell(Q)=2^{j-s},\qquad
    Q\subseteq\tfrac12J(Q),\qquad
    \operatorname{supp}b_Q\subseteq Q,\\
    \int_Qb_Q=0,\qquad
    \|b_Q\|_{L^1}\leq C_d\|b\|_{\dot{\mathcal X}_1(\mathcal Q)}|Q|.
\end{gathered}
\]
Moreover, \eqref{eq:packing-bQ-revised} gives
\[
    \sum_{j\in\mathbb Z}
    \sum_{Q\in\mathfrak V_{j-s}^{n,\iota}}
       \|b_Q\|_{L^1}
    \lesssim_d
    2^{-2(n-1)}\|b\|_{\dot{\mathcal X}_1(\mathcal Q)}|Q_0|.
\]
The cutoff estimate \eqref{4.5}, the angular decomposition,
and the multipliers \(G_v^s\) also agree with those in the cited argument.
These are precisely the structural hypotheses used there. The signs
are harmless, and the occupancy parameter \(n\) enters only through
the preceding packing estimate.

\begin{prop}[Endpoint estimates for the microlocal pieces]
\label{prop:microlocal-endpoints}
There exist constants \(\delta_3,\delta_4>0\), depending only on the
dimension and on the auxiliary parameters, such that
\[
\begin{aligned}
 \|\mathcal I_1\|_{L^1(\mathbb R^d)}
 &\lesssim_d
 2^{-2(n-1)}2^{-\kappa s}
 \|\Omega\|_{L^\infty(S^{d-1})}
 \|b\|_{\dot{\mathcal X}_1(\mathcal Q)}|Q_0|,
 &&\text{\rm(i)}\\
 \|\mathcal I_2\|_{L^1(\mathbb R^d)}
 &\lesssim_d
 2^{-2(n-1)}2^{-(1-\kappa)s}
 \|\Omega\|_{L^\infty(S^{d-1})}
 \|b\|_{\dot{\mathcal X}_1(\mathcal Q)}|Q_0|,
 &&\text{\rm(ii)}\\
 \|\mathcal I_3\|_{L^2(\mathbb R^d)}
 &\lesssim_d
 2^{-(n-1)}2^{-\delta_3s}
 \|\Omega\|_{L^\infty(S^{d-1})}
 \|b\|_{\dot{\mathcal X}_1(\mathcal Q)}|Q_0|^{1/2},
 &&\text{\rm(iii)}\\
 \|\mathcal I_4\|_{L^2(\mathbb R^d)}
 &\lesssim_d
 2^{-(n-1)}2^{-\delta_4s}
 \|\Omega\|_{L^\infty(S^{d-1})}
 \|b\|_{\dot{\mathcal X}_1(\mathcal Q)}|Q_0|^{1/2}.
 &&\text{\rm(iv)}
\end{aligned}
\]
Consequently, if
\[
    \mathcal I_{34}:=\mathcal I_3+\mathcal I_4,
    \qquad
    \delta_0:=\min\{\delta_3,\delta_4\},
\]
then
\[
 \|\mathcal I_{34}\|_{L^2(\mathbb R^d)}
 \lesssim_d
 2^{-(n-1)}2^{-\delta_0s}
 \|\Omega\|_{L^\infty(S^{d-1})}
 \|b\|_{\dot{\mathcal X}_1(\mathcal Q)}|Q_0|^{1/2}.
\]
\end{prop}

\begin{proof}
For the first two estimates, we use the corresponding single-atom
bounds from \cite[(4.19) and (4.23)]
{lai2025weak11estimatemaximal}.  In the present notation they read
\[
\begin{aligned}
 \bigl\|
   (T_jb_Q)
   (\chi_{J(Q)^\iota}
    -\widetilde\chi_{J(Q)^\iota})
 \bigr\|_{L^1}
 &\lesssim_d
 2^{-\kappa s}
 \|\Omega\|_{L^\infty(S^{d-1})}\|b_Q\|_{L^1},\\
 \bigl\|
   P_{j-s\kappa}
   [(T_jb_Q)\widetilde\chi_{J(Q)^\iota}]
 \bigr\|_{L^1}
 &\lesssim_d
 2^{-(1-\kappa)s}
 \|\Omega\|_{L^\infty(S^{d-1})}\|b_Q\|_{L^1}.
\end{aligned}
\]
The signs have modulus one and therefore do not affect these
estimates.  Summing in \(j\) and \(Q\), and then applying
\eqref{eq:packing-bQ-revised}, gives (i) and (ii).

For (iii), we give the localized form of the sector argument
in \cite[Section~4.6]{lai2025weak11estimatemaximal}.
Define
\[
    B_{j-s}:=
    \sum_{Q\in\mathfrak V_{j-s}^{n,\iota}}|b_Q|
\]
and
\[
    F_v:=
    \sum_j\sum_{Q\in\mathfrak V_{j-s}^{n,\iota}}
    (I-P_{j-s\kappa})
    \bigl[
      \varepsilon_{J(Q)}(T_j^vb_Q)
      \widetilde\chi_{J(Q)^\iota}
    \bigr].
\]
Since $P_{j-s\kappa}$ and $G_v^s$ commute,
$\mathcal I_3=\sum_vG_v^sF_v$.

For a sufficiently large fixed dimensional constant $C$,
let
\[
    \mathscr T_j^v
    :=
    \left\{
      te_v^s+u:
      u\perp e_v^s,\ 
      |t|\leq C2^j,\ 
      |u|\leq C2^{j-s\gamma}
    \right\},
    \qquad
    H_j^v:=2^{-jd}\mathbf1_{\mathscr T_j^v}.
\]
The support of $K_j^v$, the inequality $\kappa>\gamma$,
and $0\leq\widetilde\chi_{J(Q)^\iota}\leq1$ imply
\[
    |F_v|
    \lesssim_d
    \|\Omega\|_{L^\infty(S^{d-1})}\sum_jH_j^v*B_{j-s}.
\]
The kernels $H_j^v$ are even. For $i\leq j$,
\[
    H_j^v*H_i^v
    \lesssim_d
    2^{-s\gamma(d-1)}2^{-jd}
    \mathbf1_{2\mathscr T_j^v}.
\]

For each fixed $x,j,v$, the disjointness and size
bounds of the stopping atoms give
\[
\begin{aligned}
    \sum_{i\leq j}
    \int_{x+2\mathscr T_j^v}B_{i-s}(y)\,dy
    \lesssim_d
    \|b\|_{\dot{\mathcal X}_1(\mathcal Q)}\sum_{\substack{
        i\leq j,\ Q\in\mathfrak V_{i-s}^{n,\iota}\\
        Q\cap(x+2\mathscr T_j^v)\neq\emptyset}}
      |Q|
    \lesssim_d
    \|b\|_{\dot{\mathcal X}_1(\mathcal Q)}\,2^{jd-s\gamma(d-1)}.
\end{aligned}
\]
Indeed, every cube in this sum has side length at most
$2^{j-s}\leq2^{j-s\gamma}$, and hence is contained
in a fixed dilation of the tube. The cubes are
pairwise disjoint.

Expanding the square and ordering the scales, we obtain
\[
\begin{aligned}
    \|F_v\|_2^2
    &\lesssim_d
    \|\Omega\|_{L^\infty(S^{d-1})}^2\sum_j
    \int B_{j-s}(x)
       \sum_{i\leq j}
       (H_j^v*H_i^v*B_{i-s})(x)\,dx\\
    &\lesssim_d
    2^{-2s\gamma(d-1)} \|\Omega\|_{L^\infty(S^{d-1})}^2 \|b\|_{\dot{\mathcal X}_1(\mathcal Q)}\sum_j\sum_{Q\in\mathfrak V_{j-s}^{n,\iota}}\|b_Q\|_1.
\end{aligned}
\]
These estimates are uniform in the signs.

The separation of the directions $e_v^s$ gives
\[
    \sup_{\xi\neq0}
    \sum_v
    \left|
      \phi\left(
        2^{s\gamma}
        \left\langle e_v^s,\frac{\xi}{|\xi|}\right\rangle
      \right)
    \right|^2
    \lesssim_d 2^{s\gamma(d-2)}.
\]
Thus Plancherel's theorem, Cauchy--Schwarz, and
$\#\Theta_s\lesssim_d2^{s\gamma(d-1)}$ yield
\[
\begin{aligned}
    \|\mathcal I_3\|_2^2
    \lesssim_d
    2^{s\gamma(d-2)}\sum_v\|F_v\|_2^2
    \lesssim_d
    2^{-s\gamma} \|\Omega\|_{L^\infty(S^{d-1})}^2\|b\|_{\dot{\mathcal X}_1(\mathcal Q)}\sum_j\sum_{Q\in\mathfrak V_{j-s}^{n,\iota}}\|b_Q\|_1.
\end{aligned}
\]
Finally, \eqref{eq:packing-bQ-revised} gives
\[
    \|\mathcal I_3\|_2
    \lesssim_d
    2^{-(n-1)}2^{-s\gamma/2}
    \|\Omega\|_{L^\infty(S^{d-1})}
    \|b\|_{\dot{\mathcal X}_1(\mathcal Q)}
    |Q_0|^{1/2}.
\]
This proves (iii) with $\delta_3=\gamma/2$.

For completeness, we give the short transfer argument for (iv).
It also explains precisely how the parameters and the localized bad
functions correspond to those in
\cite[Section~5]{lai2025weak11estimatemaximal}.  Put
$
 N_d:=\left\lfloor\frac d2\right\rfloor+1.
$
Choose
$
 \kappa<\varepsilon_0<\frac{1+\kappa}{2},
$
then choose \(\gamma>0\) sufficiently small with
\[
 \gamma<\min\{\kappa,\varepsilon_0-\kappa\}.
\]
Finally choose an integer \(N_1\) sufficiently large.  Define
\[
\begin{aligned}
 \rho_1
 &:=
 (\varepsilon_0-\kappa-\gamma)N_1-2\gamma N_d,\\
 \rho_2
 &:=1-\varepsilon_0,\qquad
 \rho_3:=1-2\varepsilon_0+\kappa,\\
 \alpha
 &:=\gamma N_d+\frac23\gamma(d-1).
\end{aligned}
\]
The choices can be made in this order so that
\[
 \vartheta_\ell
 :=\frac{\rho_\ell}{4}-\frac{3\alpha}{4}>0,
 \qquad \ell=1,2,3.
\]

The estimates in
\cite[Section~5]{lai2025weak11estimatemaximal}
use the scale, support, cancellation, and $L^1$-size
conditions on the atoms, together with the
disjointness of their supporting cubes.
These properties have been verified above.
The spatial cutoffs satisfy the required derivative
bounds by \eqref{4.5}. Hence every occurrence of the
Calder\'on--Zygmund height in the estimates of
\cite[Section~5]{lai2025weak11estimatemaximal} is replaced by
\(\|b\|_{\dot{\mathcal X}_1(\mathcal Q)}\),
while the angular level-set factor can be estimated directly. Indeed,
for every angular cell \(E_v^s\),
\[
    \int_{E_v^s}|\Omega(\theta)|\,d\sigma(\theta)
    \leq
    \|\Omega\|_{L^\infty(S^{d-1})}\sigma(E_v^s)
    \lesssim_d
    2^{-s\gamma(d-1)}
    \|\Omega\|_{L^\infty(S^{d-1})}.
\]
In the estimates of \cite[Section~5]
{lai2025weak11estimatemaximal}, we bound the angular integrals
directly by the preceding inequality. This replaces the
level-set bound by $\|\Omega\|_\infty$ and removes the factor
$2^{s\eta}$ from the resulting estimates.
Moreover, the factor \(\lambda^2C_\Omega^{-2}\) in Lai's \(L^3\)
size estimate is replaced by
\(\|b\|_{\dot{\mathcal X}_1(\mathcal Q)}^2\), because the present atoms
satisfy
\[
    \|b_Q\|_{L^1}
    \leq C_d\|b\|_{\dot{\mathcal X}_1(\mathcal Q)}|Q|.
\]
Set $\rho_{\ast}:=\min_{1\leq \ell \leq 3} \rho_{\ell}$.
With these substitutions, the estimates in
\cite[Lemmas~5.1--5.3]{lai2025weak11estimatemaximal},
combined and summed as in
\cite[(5.16)--(5.17)]{lai2025weak11estimatemaximal},
give the $L^1$ bound below. The $L^3$ bound follows
from the proof of
\cite[Lemma~4.8]{lai2025weak11estimatemaximal}
with the same substitutions:
\[
\begin{aligned}
 \|\mathcal I_4\|_{L^1}
 &\lesssim_d
 \sum_{\ell=1}^3 2^{-\rho_\ell s}
 \|\Omega\|_{L^\infty(S^{d-1})}
 \sum_{j\in\mathbb Z}
 \sum_{Q\in\mathfrak V_{j-s}^{n,\iota}}
 \|b_Q\|_{L^1}\\
 &\lesssim_d
 2^{-\rho_\ast s}
 \|\Omega\|_{L^\infty(S^{d-1})}
 \sum_{j\in\mathbb Z}
 \sum_{Q\in\mathfrak V_{j-s}^{n,\iota}}
  \|b_Q\|_{L^1},
\end{aligned}
\]
and
\[
 \|\mathcal I_4\|_{L^3}
 \lesssim_d
 2^{\alpha s}
 \|\Omega\|_{L^\infty(S^{d-1})}
 \|b\|_{\dot{\mathcal X}_1(\mathcal Q)}^{2/3}
 \left(
   \sum_{j\in\mathbb Z}
   \sum_{Q\in\mathfrak V_{j-s}^{n,\iota}}
   \|b_Q\|_{L^1}
 \right)^{1/3}.
\]
Using \eqref{eq:packing-bQ-revised} in these two estimates yields 
\[
\begin{aligned}
 \|\mathcal I_4\|_{L^1}
 &\lesssim_d
 2^{-\rho_\ast s}
 2^{-2(n-1)}
 \|\Omega\|_{L^\infty(S^{d-1})}\|b\|_{\dot{\mathcal X}_1(\mathcal Q)}|Q_0|,\\
 \|\mathcal I_4\|_{L^3}
 &\lesssim_d
 2^{\alpha s}2^{-2(n-1)/3}
 \|\Omega\|_{L^\infty(S^{d-1})}\|b\|_{\dot{\mathcal X}_1(\mathcal Q)}|Q_0|^{1/3}.
\end{aligned}
\]
Finally, interpolate between \(L^1\) and \(L^3\), using
$
 \frac12=\frac14\cdot1+\frac34\cdot\frac13.
$
The resulting \(s\)-exponent is
\(\delta_4=(\rho_\ast-3\alpha)/4>0\), and
$
    \bigl(2^{-2(n-1)}\bigr)^{1/4}
    \bigl(2^{-2(n-1)/3}\bigr)^{3/4}
    =2^{-(n-1)}.
$
This proves (iv).  The final estimate for \(\mathcal I_{34}\) follows from
the triangle inequality.
\end{proof}
\begin{rem}[The separation parameter]
\label{rem:Lai-microlocal-transfer}
The preceding estimates hold for every integer $s>1$.
Indeed, the dyadic geometry gives
$Q\subseteq\tfrac12J(Q)$, so the physical localization
preserves each atom and its cancellation. Moreover,
\[
    2^{j-s}\leq 2^{j-s\kappa}
    \leq 2^{j-s\gamma}\leq 2^j,
    \qquad 0<\gamma<\kappa<1.
\]
Thus smoothing at scale $2^{j-s\kappa}$ preserves the
required sector support bounds, while the boundary-layer
and cancellation estimates give the factors
$2^{-\kappa s}$ and $2^{-(1-\kappa)s}$, respectively.

For the complementary microlocal term, wherever the
multiplier of $I-G_v^s$ is nonzero, every $\theta\in E_v^s$
satisfies
\[
    \left|
      \left\langle\theta,\frac{\xi}{|\xi|}\right\rangle
    \right|
    \geq 2^{1-s\gamma}-2^{-s\gamma-2}
    \geq 2^{-s\gamma}.
\]
Consequently, the radial integrations by parts used in
\cite[Section~5]{lai2025weak11estimatemaximal}
apply with the same scale factors for every integer $s>1$.
All auxiliary parameters are fixed independently of $p$.

For $s=1$, we retain
$b_{Q,J}=b_Q\mathbf1_{\frac12J}$ and use the
positive-kernel estimates, which require only support
and $L^1$-size.
\end{rem}

\subsection{Positive-kernel \(L^3\) estimates}
\label{subsec:Lr-estimates-interpolation}

We prove the positive-kernel $L^3$ bounds used below
to obtain decay for signed sums in $L^2$.
The proof follows the moment argument in
\cite[Lemmas~4.6 and~4.7]{lai2025weak11estimatemaximal}.

\begin{lem}\label{lem:positive-kernel-L3}
For every integer $s>1$,
\[
\begin{aligned}
    \max\Bigg\{
    \left\|
       \sum_j\sum_{Q\in\mathfrak V_{j-s}^{n,\iota}}
          |T_jb_Q|
      \right\|_3,&
    \left\|
       \sum_j\sum_{Q\in\mathfrak V_{j-s}^{n,\iota}}
          P_{j-s\kappa}[|T_jb_Q|]
      \right\|_3
    \Bigg\}\\
    &\leq
    C_d2^{-2(n-1)/3}
       \|\Omega\|_\infty
       \|b\|_{\dot{\mathcal X}_1(\mathcal Q)}
       |Q_0|^{1/3}.
\end{aligned}
\]
The constant is independent of $p,s,n$ and the stopping
collection.
\end{lem}

\begin{proof}
Set
\[
    B_{j-s}
    :=\sum_{Q\in\mathfrak V_{j-s}^{n,\iota}}|b_Q|.
\]
For either positive sum in the statement, it suffices
to estimate $\|S\|_3$, where
\[
    S(x):=\sum_j(H_{j,s}*B_{j-s})(x)
\]
and the nonnegative kernels satisfy
\[
    H_{j,s}(x)
    \leq C_d\|\Omega\|_\infty
       2^{-jd}\mathbf1_{B(0,C2^j)}(x).
\]
For the first sum, take $H_{j,s}=|K_j|$; for the
second, take $H_{j,s}=P_{j-s\kappa}|K_j|$.
The latter satisfies the same bound because the
smoothing kernel is nonnegative, has integral one,
and is supported in a ball of radius at most $C2^j$.

We first work with finite sums.
Expanding the cube and ordering the scales gives
\[
    \|S\|_3^3
    \leq
    6\sum_{j_1\geq j_2\geq j_3}
       \int_{\mathbb R^d}
       \prod_{\nu=1}^3
          (H_{j_\nu,s}*B_{j_\nu-s})(x)\,dx.
\]
The kernel size and support estimates imply
\[
\begin{aligned}
    \int_{\mathbb R^d}
       \prod_{\nu=1}^3H_{j_\nu,s}(x-y_\nu)\,dx
       \leq
    C_d\|\Omega\|_\infty^3
       2^{-j_1d}2^{-j_2d}
       \mathbf1_{\{|y_1-y_2|\leq C2^{j_1}\}}
       \mathbf1_{\{|y_2-y_3|\leq C2^{j_2}\}}.
\end{aligned}
\]
For every fixed $j$ and $y$, disjointness of the stopping
cubes and their atomic size bounds give
\[
    \sum_{i\leq j}
       \int_{|z-y|\leq C2^j}B_{i-s}(z)\,dz
    \leq
    C_d\|b\|_{\dot{\mathcal X}_1(\mathcal Q)}2^{jd}.
\]
Indeed, every stopping cube in this sum has side length
at most $2^j$ and is contained in a fixed enlargement
of the ball.

After applying Fubini, use this estimate first for
$(j_3,y_3)$ and then for $(j_2,y_2)$. It follows that
\[
    \|S\|_3^3
    \leq
    C_d\|\Omega\|_\infty^3
       \|b\|_{\dot{\mathcal X}_1(\mathcal Q)}^2
       \sum_j\int_{\mathbb R^d}B_{j-s}.
\]
By \eqref{eq:packing-bQ-revised},
\[
    \sum_j\int_{\mathbb R^d}B_{j-s}
    \leq
    C_d2^{-2(n-1)}
       \|b\|_{\dot{\mathcal X}_1(\mathcal Q)}|Q_0|.
\]
Taking the cube root proves the estimate.
The general case follows by monotone convergence.
\end{proof}

Since
$|\mathbf1_{J^\iota}-\widetilde\chi_{J^\iota}|\leq1$
and $0\leq\widetilde\chi_{J^\iota}\leq1$,
Lemma~\ref{lem:positive-kernel-L3} implies
\[
    \max_{\nu=1,2}\|\mathcal I_\nu\|_3
    \leq
    C_d2^{-2(n-1)/3}
       \|\Omega\|_\infty
       \|b\|_{\dot{\mathcal X}_1(\mathcal Q)}
       |Q_0|^{1/3},
\]
uniformly in the fixed signs.

\subsection{Completion of the bounded-kernel argument}
\label{subsec:completion-bounded-kernel-case}
Fix $1<p\leq2$ and retain the occupancy decomposition
from Section~\ref{subsec:occupancy-rademacher-menshov}
throughout this subsection.
For fixed $s,\vec w,\iota,n$ and the functions $\beta_u$
defined above, set
\[
    \mathcal R_{s,n}^{\sharp}b(x)
    :=\max_{0\leq v_0\leq v_1\leq u_0}
         \left|\sum_{u=v_0+1}^{v_1}\beta_u(x)\right|.
\]
The geometric reduction in
Section~\ref{subsec:occupancy-rademacher-menshov} bounds every
linearized contribution by $\mathcal R_{s,n}^{\sharp}b$.

For $s>1$, Proposition~\ref{prop:microlocal-endpoints}
and Lemma~\ref{lem:positive-kernel-L3} give, for $\nu=1,2$,
\[
    \|\mathcal I_\nu\|_2
    \leq\|\mathcal I_\nu\|_1^{1/4}
         \|\mathcal I_\nu\|_3^{3/4}
    \leq C_d \|\Omega\|_\infty \|b\|_{\dot{\mathcal X}_1(\mathcal Q)} (2^{-2(n-1)}|Q_0|)^{1/2}2^{-a_\nu s/4},
\]where $a_1=\kappa, a_2=1-\kappa$.
Combining this with the $L^2$ estimate for $\mathcal I_{34}$, and setting
\[
    \delta_2=\min\{\kappa/4,(1-\kappa)/4,\delta_0\}>0,
\]
we obtain the uniform signed estimate
\[
    \sup_{\varepsilon_u\in\{-1,1\}}
    \left\|\sum_{u=1}^{u_0}\varepsilon_u\beta_u\right\|_2
    \leq C_d \|\Omega\|_\infty \|b\|_{\dot{\mathcal X}_1(\mathcal Q)} (2^{-2(n-1)}|Q_0|)^{1/2}2^{-\delta_2s}.
\]
Since
\begin{equation}\label{eq:RM-new-log}
    \log(2+u_0)\leq C_d(1+s),
\end{equation}
the Rademacher--Menshov inequality, applied to prefixes
and then to their differences, yields
\begin{equation}\label{eq:maximal-L2-improved}
    \|\mathcal R_{s,n}^{\sharp}b\|_2
    \leq C_d(1+s)2^{-(n-1)}2^{-\delta_2s}
       \|\Omega\|_\infty
       \|b\|_{\dot{\mathcal X}_1(\mathcal Q)}
       |Q_0|^{1/2}.
\end{equation}

The definition of $\mathcal Y_p$ and the disjointness of the stopping
cubes give
\begin{equation}\label{eq:Yp-local-Lp}
    \|h\|_{L^p(3Q_0)}
    \leq C_d|Q_0|^{1/p}\|h\|_{\mathcal Y_p(\mathcal Q)}.
\end{equation}
Indeed, use the $L^\infty$ bound outside the stopping shadow and
$\|h\mathbf1_L\|_p\leq C_d|L|^{1/p}\|h\|_{\mathcal Y_p}$
on each stopping cube $L$.
To obtain the quantitative testing bound, we estimate the pairing
directly. First, for every $g\in\mathcal Y_1(\mathcal Q)$ and
every $s\geq1$, the positive-kernel majorant gives
\begin{equation}\label{eq:positive-Y1-pairing}
\begin{aligned}
    \frac1{|Q_0|}
    \int_{\mathbb R^d}
       \mathcal R_{s,n}^{\sharp}b(x)|g(x)|\,dx
    \leq{}&
    C_d2^{-2(n-1)}
    \|\Omega\|_\infty
    \|b\|_{\dot{\mathcal X}_1(\mathcal Q)}
    \|g\|_{\mathcal Y_1(\mathcal Q)}.
\end{aligned}
\end{equation}
By the triangle inequality and the physical localization
identities,
\[
    \mathcal R_{s,n}^{\sharp}b
    \leq
    \begin{cases}
    \displaystyle
    \sum_j\sum_{Q\in\mathfrak V_{j-s}^{n,\iota}}
       |T_jb_Q|,
       &s>1,\\[6pt]
    \displaystyle
    \sum_j
    \sum_{(Q,J)\in\widetilde{\mathfrak V}_{j-1}^{n,\iota}}
       |T_jb_{Q,J}|,
       &s=1.
    \end{cases}
\]
We estimate the pairing of each term separately.
Suppose that $\operatorname{supp}a\subseteq Q$ and
$\ell(Q)=2^{j-s}$.
For $y\in Q$ and $z\in\widehat Q$, the inequality $s\geq1$
ensures that $|y-z|_\infty\leq C_d2^j$.
Thus there is a cube $R_{j,z}$ containing $z$ and
the support of $K_j(\cdot-y)$, with
$|R_{j,z}|\leq C_d2^{jd}$.
Consequently,
\[
\begin{aligned}
    \int_{\mathbb R^d}|K_j(x-y)||g(x)|\,dx
    \leq
    C_d\|\Omega\|_\infty
       2^{-jd}\int_{R_{j,z}}|g(x)|\,dx
    \leq C_d\|\Omega\|_\infty Mg(z).
\end{aligned}
\]
Taking the infimum over $z\in\widehat Q$ and using Fubini,
we obtain
\[
\begin{aligned}
    \int_{\mathbb R^d}|T_ja(x)||g(x)|\,dx
    &\leq
    C_d\|\Omega\|_\infty\|a\|_1
       \inf_{z\in\widehat Q}Mg(z)\\
    &\leq
    C_d\|\Omega\|_\infty\|a\|_1
       \|g\|_{\mathcal Y_1(\mathcal Q)}.
\end{aligned}
\]
Apply this estimate with $a=b_Q$ for $s>1$ and with
$a=b_{Q,J}$ for $s=1$.
Summing the atomic masses by
\eqref{eq:packing-bQ-revised} and
\eqref{eq:packing-truncated-atoms-s1}, and dividing by
$|Q_0|$, proves \eqref{eq:positive-Y1-pairing}.
No norm estimate for the maximal operator is used here.

Fix $1<p\leq2$ and $s>1$.
By homogeneity, assume
$\|h\|_{\mathcal Y_p(\mathcal Q)}=1$.
For $\lambda>0$, write
\[
    h_{\leq\lambda}
       :=h\mathbf1_{\{|h|\leq\lambda\}},
    \qquad
    h_{>\lambda}
       :=h\mathbf1_{\{|h|>\lambda\}}.
\]
By \eqref{eq:Yp-local-Lp},
\[
\begin{aligned}
    \|h_{\leq\lambda}\|_2^2
    \leq
    \lambda^{2-p}\int_{3Q_0}|h|^p
    \leq C_d\lambda^{2-p}|Q_0|.
\end{aligned}
\]
Moreover,
\[
    M(h_{>\lambda})
    \leq
    \lambda^{1-p}(M_ph)^p,
\]
and hence
\[
    \|h_{>\lambda}\|_{\mathcal Y_1(\mathcal Q)}
    \leq \lambda^{1-p}.
\]
To verify the part of this norm outside the stopping shadow,
observe that $|h|\leq1$ there. If $\lambda\geq1$,
$h_{>\lambda}$ vanishes there; if $0<\lambda<1$, its
absolute value is at most $1\leq\lambda^{1-p}$.

Using Cauchy--Schwarz and
\eqref{eq:maximal-L2-improved} for $h_{\leq\lambda}$,
and \eqref{eq:positive-Y1-pairing} for $h_{>\lambda}$,
we obtain
\[
\begin{aligned}
    \frac1{|Q_0|}
      \int_{\mathbb R^d}
         \mathcal R_{s,n}^{\sharp}b\,|h|
    \leq{}
    C_d\|\Omega\|_\infty
       \|b\|_{\dot{\mathcal X}_1(\mathcal Q)}
    \Big[
       2^{-(n-1)}(1+s)2^{-\delta_2s}
          \lambda^{1-p/2}
       +2^{-2(n-1)}\lambda^{1-p}
    \Big].
\end{aligned}
\]
Choose
\[
    \lambda
    :=\left(\frac{2^{\delta_2s}}{1+s}\right)^{2/p}.
\]
Since $2(p-1)/p=2/p'$, it follows that
\[
\begin{aligned}
    \frac1{|Q_0|}
      \int_{\mathbb R^d}
         \mathcal R_{s,n}^{\sharp}b\,|h|
    \leq{}&
    C_d2^{-(n-1)}
       (1+s)^{2/p'}2^{-2\delta_2s/p'}
       \|\Omega\|_\infty
       \|b\|_{\dot{\mathcal X}_1(\mathcal Q)}.
\end{aligned}
\]
If $p'>s$, then
\[
    (1+s)^{2/p'}\leq(2p')^{2/p'}\leq C.
\]
If $p'\leq s$, then $2/p'\leq1$ gives
\[
    (1+s)^{2/p'}
    \leq(2p')^{2/p'}(s/p')^{2/p'}
    \leq C\,s/p'.
\]
Thus, in either case,
\[
    (1+s)^{2/p'}2^{-2\delta_2s/p'}
    \leq C_d2^{-\delta_2s/p'},
\]
where we used
$\sup_{t\geq1}t\,2^{-\delta_2t}<\infty$
in the second case.
Restoring the normalization of $h$ gives
\begin{equation}\label{eq:paired-fixed-separation-linear}
\begin{aligned}
    \frac1{|Q_0|}
    \int_{\mathbb R^d}
       \mathcal R_{s,n}^{\sharp}b\,|h|
    \leq{}
    C_d2^{-(n-1)}2^{-c_ds/p'}
       \|\Omega\|_\infty
       \|b\|_{\dot{\mathcal X}_1(\mathcal Q)}
       \|h\|_{\mathcal Y_p(\mathcal Q)},
\end{aligned}
\end{equation}
where $c_d=\delta_2/2$ is admissible.

For $s=1$, use \eqref{eq:positive-Y1-pairing} and
$\|h\|_{\mathcal Y_1}\leq\|h\|_{\mathcal Y_p}$.
This proves \eqref{eq:paired-fixed-separation-linear}
also for $s=1$, after increasing the dimensional constant.

Summing \eqref{eq:paired-fixed-separation-linear} in $n,s$
and the finitely many grid and child indices yields
\begin{equation}\label{eq:first-testing-linear-pprime}
    |\mathcal Q_{t_1}^{t_2}(b,h)|
    \leq
    C_dp'\|\Omega\|_\infty
       \|b\|_{\dot{\mathcal X}_1(\mathcal Q)}
       \|h\|_{\mathcal Y_p(\mathcal Q)}.
\end{equation}
Indeed,
\[
    \sum_{n\geq1}2^{-(n-1)}=2,
    \qquad
    \sum_{s\geq1}2^{-c_ds/p'}
    =\frac1{2^{c_d/p'}-1}
    \leq\frac{p'}{c_d\log2}.
\]

It remains to verify the second localized testing estimate.
By homogeneity, the first estimate in
\cite[(1.9)]{MR4245601} gives a weak \(L^r\) bound
with constant \(C_d\max\{r,r'\}\|\Omega\|_\infty\).
The reciprocal truncations used there give the same bound
for \(T_\Omega^\ast\): applying the estimate to \(f(a\,\cdot)\)
and rescaling yields the bound for
\[
    \sup_{\varepsilon>0}
    \left|
      \int_{\varepsilon<|y|<a^2/\varepsilon}
        \frac{\Omega(y/|y|)}{|y|^d}f(x-y)\,dy
    \right|,
\]
uniformly in \(a>0\).
Letting \(a\to\infty\) for bounded, compactly supported \(f\),
and using Fatou's lemma for distribution functions, gives
\[
    \|T_\Omega^\ast f\|_{L^{r,\infty}}
    \leq C_dr\|\Omega\|_\infty\|f\|_r,
    \qquad 2\leq r<\infty.
\]
The comparison with radial truncations established in
Section~\ref{subsec:dyadic-abstract-reduction}, together with
the $L^r$ boundedness of the Hardy--Littlewood maximal operator,
therefore yields
\[
    \|T_{\Omega,\ast}f\|_{L^{r,\infty}}
    \leq C_dr\|\Omega\|_\infty\|f\|_r,
    \qquad 2\leq r<\infty.
\]

For every measurable set $E$ of finite measure,
the weak-to-strong embedding gives
\[
    \|F\|_{L^{p'}(E)}
    \leq
    2^{1/p'}|E|^{1/(2p')}
       \|F\|_{L^{2p',\infty}(\mathbb R^d)}.
\]
Applying the preceding weak-type estimate at $r=2p'$,
we obtain, for every cube $R$,
\[
\begin{aligned}
    \|T_{\Omega,\ast}(h\mathbf1_R)\|_{L^{p'}(3R)}
    &\leq
    C_dp'\|\Omega\|_\infty
       |3R|^{1/(2p')}\|h\mathbf1_R\|_{2p'}\\
    &\leq
    C_dp'\|\Omega\|_\infty
       |R|^{1/p'}\|h\|_\infty.
\end{aligned}
\]

For each stopping cube $L\in\mathcal Q$, the definition of
$\mathcal Y_p(\mathcal Q)$ implies
\[
    \|b\mathbf1_{3L}\|_p
    \leq
    C_d|L|^{1/p}
       \|b\|_{\dot{\mathcal X}_p(\mathcal Q)}.
\]
Indeed, for every $z\in\widehat L$, a cube containing
both $3L$ and $z$ can be chosen with measure at most
$C_d|L|$. Taking the infimum of $M_pb(z)$ over
$\widehat L$ proves the estimate.
Also, \eqref{eq:Yp-local-Lp} gives
\[
    \|b\|_{L^p(3Q_0)}
    \leq
    C_d|Q_0|^{1/p}
       \|b\|_{\dot{\mathcal X}_p(\mathcal Q)}.
\]

The kernel support condition implies
\[
    \operatorname{supp}
       T[K^\Omega]_{t_1}^{\,t_2\wedge s_R}
          (h\mathbf1_R)
    \subseteq3R.
\]
Thus the definition of the localized form and
H\"older's inequality yield
\[
\begin{aligned}
    |\mathcal Q_{t_1}^{t_2}(h,b)|
    \leq \frac1{|Q_0|}\Bigg[
       \|T_{\Omega,\ast}(h\mathbf1_{Q_0})\|_{L^{p'}(3Q_0)}
          \|b\|_{L^p(3Q_0)}
       +\sum_{\substack{L\in\mathcal Q\\L\subset Q_0}}
          \|T_{\Omega,\ast}(h\mathbf1_L)\|_{L^{p'}(3L)}
          \|b\|_{L^p(3L)}
    \Bigg].
\end{aligned}
\]
Combining the preceding estimates and using the
disjointness of the stopping cubes, we conclude that
\begin{equation}\label{eq:second-testing-linear-pprime}
\begin{aligned}
    |\mathcal Q_{t_1}^{t_2}(h,b)|
    &\leq
    C_dp'\|\Omega\|_\infty
       \|h\|_{\mathcal Y_\infty(\mathcal Q)}
       \|b\|_{\dot{\mathcal X}_p(\mathcal Q)}
       \frac{|Q_0|+
          \sum_{L\in\mathcal Q,\ L\subset Q_0}|L|}
            {|Q_0|}\\
    &\leq
    C_dp'\|\Omega\|_\infty
       \|h\|_{\mathcal Y_\infty(\mathcal Q)}
       \|b\|_{\dot{\mathcal X}_p(\mathcal Q)}.
\end{aligned}
\end{equation}

Apply Theorem~\ref{thm3.6} with the fixed exponent $r=2$.
The global $L^2$ maximal-truncation norm is bounded by
$C_d\|\Omega\|_\infty$, while
\eqref{eq:first-testing-linear-pprime} and
\eqref{eq:second-testing-linear-pprime} bound both localized
testing constants by $C_dp'\|\Omega\|_\infty$.
The abstract theorem introduces no further dependence on $p$.
Consequently,
\[
    \|T_{\Omega,\ast}\|_{(1,p)\text{-}\mathrm{sparse}}
    \leq C_dp'\|\Omega\|_{L^\infty(S^{d-1})},
    \qquad 1<p\leq2.
\]

To pass from the dyadic maximal truncation to the radial one, let $f$
and $g$ be bounded and compactly supported. The pointwise comparison
at the beginning of this section and the standard sparse bound for the
Hardy--Littlewood maximal operator give
\begin{align*}
    \bigl|\langle T_{\Omega}^{\ast}f,g\rangle\bigr|
    &\lesssim_d
    \bigl\langle
        T_{\Omega,\ast}f,|g|
    \bigr\rangle
    +
    \|\Omega\|_{L^\infty(S^{d-1})}
    \langle M|f|,|g|\rangle
    \\
    &\lesssim_d
    p'\|\Omega\|_{L^\infty(S^{d-1})}
    \sup_{\mathcal S}
    \Lambda_{\mathcal S,1,p}(f,g),
\end{align*}
where the supremum is taken over all \(1/2\)-sparse collections.
For completeness, this supremum is finite. Indeed, if
\(\{E_Q:Q\in\mathcal S\}\) witnesses the sparseness of \(\mathcal S\),
then, choosing \(1<a<p'\),
\begin{equation}\label{eq:sparse-supremum-finite}
\begin{aligned}
    \Lambda_{\mathcal S,1,p}(f,g)
    \leq
    2\sum_{Q\in\mathcal S}
       \int_{E_Q}M|f|\,M_p|g|
    \leq
    2\int_{\mathbb R^d}M|f|\,M_p|g|
    \lesssim_{a,p}
    \|f\|_{L^a}\|g\|_{L^{a'}}
    <\infty.
\end{aligned}
\end{equation}
Here we used \(a'>p\), together with H\"older's inequality and the
standard bounds for \(M\) and \(M_p\). Hence we may choose a
\(1/2\)-sparse collection \(\mathcal S\) such that
\[
    \sup_{\mathcal S'}
       \Lambda_{\mathcal S',1,p}(f,g)
    \leq
    2\Lambda_{\mathcal S,1,p}(f,g).
\]
This proves
\[
    \|T_{\Omega}^{\ast}\|_{(1,p)\text{-}\mathrm{sparse}}
    \lesssim_d
    p'\|\Omega\|_{L^\infty(S^{d-1})},
    \qquad 1<p\leq 2.
\]

Finally, let \(p\geq2\). The estimate just proved at \(p=2\) gives a
sparse \((1,2)\)-bound with a dimensional constant. Since normalized
local Lebesgue averages are increasing in their exponent,
\[
    \langle|g|\rangle_{2,Q}
    \leq
    \langle|g|\rangle_{p,Q},
    \qquad p\geq2,
\]
the same sparse family gives a \((1,p)\)-bound. In this range
\(1<p'\leq2\), so a dimensional constant is bounded by
\(C_dp'\). Therefore, for every \(1<p<\infty\),
\[
    \|T_{\Omega}^{\ast}\|_{(1,p)\text{-}\mathrm{sparse}}
    \lesssim_d
    p'\|\Omega\|_{L^\infty(S^{d-1})}.
\]

\section{Extension to angular kernels in
\texorpdfstring{\(L^{q,1}\log L\)}{Lq1 log L}}
\label{sec:Lq1logL-extension}

Throughout this section, \(1<q<\infty\).  We use the distribution
function \(\mu_\Omega\), the Lorentz norm \(L_q(\Omega)\), and the
homogeneous Lorentz--Zygmund functional introduced in
\eqref{eq:Lq1logL-functional}; in particular,
\[
    L_q(\Omega)
    \leq
    \|\Omega\|_{L^{q,1}\log L(S^{d-1})}.
\]
We continue to assume the cancellation condition
\begin{equation}\label{eq:Omega-global-cancellation}
    \int_{S^{d-1}}\Omega(\theta)\,d\sigma(\theta)=0.
\end{equation}
For an angular kernel \(\Theta\), we write
\[
    \mathcal Q_{\Theta,t_1}^{t_2}
    :=
    \mathcal Q[K^\Theta]_{t_1}^{t_2}.
\]
The cancellation condition \eqref{eq:Omega-global-cancellation} is essential. If
\(\Omega\) has nonzero mean, its constant angular component produces a
homogeneous kernel with nonzero mass on every annulus, and the maximal
partial sums are not uniformly bounded.  Thus the desired theorem for
the rough singular integral cannot be recovered by treating that
constant component as a harmless maximal-function error.

\subsection{The first localized testing estimate}
\label{subsec:first-localized-LqlogL}

\begin{prop}\label{prop:first-localized-LqlogL}
Let
\[
    \Omega\in L^{q,1}\log L(S^{d-1})
\]
satisfy \eqref{eq:Omega-global-cancellation}, and let \(q'\leq p<\infty\).
For every \(b\in\dot{\mathcal X}_1(\mathcal Q)\) and
\(h\in\mathcal Y_p(\mathcal Q)\), the following estimate
holds uniformly over stopping collections \(\mathcal Q\)
with top \(Q_0\) and bounded measurable integer-valued
truncation functions \(t_1\leq t_2\):
\begin{equation}\label{eq:first-localized-q-final}
\begin{aligned}
    \bigl|\mathcal Q_{\Omega,t_1}^{t_2}(b,h)\bigr|
    \lesssim_{d,q}{}&
    p'
    \|\Omega\|_{L^{q,1}\log L(S^{d-1})}
    \|b\|_{\dot{\mathcal X}_1(\mathcal Q)}
    \|h\|_{\mathcal Y_p(\mathcal Q)}.
\end{aligned}
\end{equation}
\end{prop}

\begin{proof}
Set
\[
    p_{\sharp}:=\max\{2,p'\}.
\]

We retain the occupancy decomposition and the geometric reduction
from Section~\ref{subsec:occupancy-rademacher-menshov}.
The measure and atomic packing bounds established there and in
\eqref{eq:packing-bQ-revised} and \eqref{eq:packing-truncated-atoms-s1}
are independent of \(p\).

For a fixed translated cube family, \(s\geq1\), \(n\geq1\), and
\(1\leq\iota\leq2^d\), denote the corresponding maximal localized
contribution by
\[
\begin{aligned}
    \mathcal A_{s,n,\iota}^{\Theta}b(x)
    &:=
    \sup_{\ell\in\mathbb Z}
    \left|
      \sum_{\substack{
        J\in\mathcal I_{s,\iota}^{\sharp,n}\\
        s_J\geq\ell}}
      T_J^\Theta b_{s_J-s}(x)\mathbf1_{J^\iota}(x)
    \right|,
\end{aligned}
\]
where \(\Theta\) is an angular kernel and
\[
    T_J^\Theta g
    :=T_{s_J}^\Theta
       \bigl(g\mathbf1_{\frac12J}\bigr).
\]
By the pointwise reduction in
Section~\ref{subsec:occupancy-rademacher-menshov}, the doubly
truncated expression is bounded by a dimensional multiple of this
maximal partial sum. The finite sums over \(\vec w\) and \(\iota\)
 will be absorbed into the
dimensional constant.

For every bounded mean-zero angular kernel $\Theta$,
the paired estimate
\eqref{eq:paired-fixed-separation-linear}, summed in the
occupancy generation, gives
\begin{equation}\label{eq:bounded-fixed-s-used}
    \frac1{|Q_0|}\sum_{n\geq1}
    \int_{\mathbb R^d}
       \mathcal A_{s,n,\iota}^{\Theta}b\,|h|
    \leq
    C_d2^{-c_ds/p_\sharp}
       \|\Theta\|_\infty
       \|b\|_{\dot{\mathcal X}_1}
       \|h\|_{\mathcal Y_p}.
\end{equation}
For $1<p\leq2$, use $p_\sharp=p'$.
For $p\geq2$, use the estimate at $p=2$ and
$\|h\|_{\mathcal Y_2}\leq\|h\|_{\mathcal Y_p}$, so that $p_\sharp=2$.
The adjacent separation $s=1$ is included by changing the constant.

If $\Omega=0$, there is nothing to prove.  We therefore assume
\(\Omega\neq0\), and set
\[
    \varepsilon:=\frac{c_d}{4p_{\sharp}},
    \qquad
    \lambda_s:=L_q(\Omega)2^{\varepsilon s},
    \qquad s\geq1.
\]
For each \(s\), define
\[
    c_s
    :=
    \int_{S^{d-1}}
       \Omega(\theta)
       \mathbf1_{\{|\Omega|\leq\lambda_s\}}(\theta)
       \,d\sigma(\theta)
\]
and decompose
\begin{align}
    \Omega_s^{\mathrm g}
    &:=
    \Omega\mathbf1_{\{|\Omega|\leq\lambda_s\}}-c_s,
    \label{eq:mean-zero-good-part}\\
    \Omega_s^{\mathrm b}
    &:=
    \Omega\mathbf1_{\{|\Omega|>\lambda_s\}}+c_s.
    \label{eq:mean-zero-bad-part}
\end{align}
Here constants are regarded as constant functions on the normalized
sphere.  Then
\[
    \Omega
    =
    \Omega_s^{\mathrm g}+\Omega_s^{\mathrm b}
\]
for every \(s\). 
Let $\mathcal Q_{\mathrm g}(b,h)$ and
$\mathcal Q_{\mathrm b}(b,h)$ denote the total contributions
obtained by inserting $\Omega_s^{\mathrm g}$ and
$\Omega_s^{\mathrm b}$, respectively, at scale separation $s$.
Moreover, by
\eqref{eq:Omega-global-cancellation},
\[
    c_s
    =
    -\int_{S^{d-1}}
       \Omega(\theta)
       \mathbf1_{\{|\Omega|>\lambda_s\}}(\theta)
       \,d\sigma(\theta),
\]
and hence
\[
    \int_{S^{d-1}}\Omega_s^{\mathrm g}\,d\sigma
    =
    \int_{S^{d-1}}\Omega_s^{\mathrm b}\,d\sigma
    =0.
\]
The normalized surface measure and H\"older's inequality give
\begin{equation}\label{eq:recentered-size-bounds}
    \|\Omega_s^{\mathrm g}\|_{L^\infty}
    \leq2\lambda_s,
    \qquad
    \|\Omega_s^{\mathrm b}\|_{L^q}
    \leq
    2\|\Omega\mathbf1_{\{|\Omega|>\lambda_s\}}\|_{L^q}.
\end{equation}
For the first inequality, use
\[
    |c_s|
    \leq
    \int_{\{|\Omega|\leq\lambda_s\}}|\Omega|\,d\sigma
    \leq\lambda_s;
\]
for the second, use the preceding representation of \(c_s\).
The constants \(c_s\) are needed because the two raw restrictions of
\(\Omega\) need not have mean zero; they restore cancellation in both
pieces without creating an additional constant-kernel term.

\medskip
\noindent
\emph{The good angular part.}
For each fixed \(s\), the contribution is linear in the angular
kernel, so the \(s\)-dependent good--bad decomposition is legitimate.

Since $\Omega_s^{\mathrm g}$ is bounded and has mean zero,
\eqref{eq:bounded-fixed-s-used}, together with
$\varepsilon=c_d/(4p_\sharp)$, gives
\begin{equation}
   \begin{aligned}
    |\mathcal Q_{\mathrm g}(b,h)|
    &\leq
    C_dL_q(\Omega)
       \sum_{s\geq1}2^{-3c_ds/(4p_\sharp)}
       \|b\|_{\dot{\mathcal X}_1}
       \|h\|_{\mathcal Y_p}\\
    &\leq
    C_dp_\sharp L_q(\Omega)
       \|b\|_{\dot{\mathcal X}_1}
       \|h\|_{\mathcal Y_p}\\
    &\leq
    C_dp'\|\Omega\|_{L^{q,1}\log L}
       \|b\|_{\dot{\mathcal X}_1}
       \|h\|_{\mathcal Y_p}.
    \label{eq:good-part-P3-corrected}
   \end{aligned} 
\end{equation}

Here we used
\[
    \sum_{s\geq1}2^{-as/p_\sharp}\leq C_ap_\sharp,
    \qquad
    p_\sharp=\max\{2,p'\}\leq2p'.
\]

\medskip
\noindent
\emph{The bad angular part.}
For this part, we do not use the Rademacher--Menshov inequality.  We
instead dominate each maximal partial sum by the corresponding
absolute sum.

To include the exceptional scale separation \(s=1\) without ambiguity,
let \(\mathscr P_{s,n,\iota}\) denote the following collection of pairs:
\[
    \mathscr P_{s,n,\iota}
    :=
    \begin{cases}
    \displaystyle
    \bigcup_{j\in\mathbb Z}
       \{(Q,J(Q)):
         Q\in\mathfrak V_{j-s}^{n,\iota}\},
       &s\geq2,\\[8pt]
    \displaystyle
    \bigcup_{j\in\mathbb Z}
       \widetilde{\mathfrak V}_{j-1}^{n,\iota},
       &s=1.
    \end{cases}
\]
For \((Q,J)\in\mathscr P_{s,n,\iota}\), put
\[
    a_{Q,J}^{(s)}
    :=
    \begin{cases}
       b_Q,&s\geq2,\\
       b_{Q,J},&s=1.
    \end{cases}
\]
In both cases,
\[
    \operatorname{supp}a_{Q,J}^{(s)}\subseteq Q,
    \qquad
    \|a_{Q,J}^{(s)}\|_{L^1}\leq\|b_Q\|_{L^1}.
\]
By \eqref{eq:packing-bQ-revised} for \(s\geq2\) and
\eqref{eq:packing-truncated-atoms-s1} for \(s=1\),
together with \(2^{-2(n-1)}=4\,2^{-2n}\), we obtain
\begin{equation}\label{eq:bad-atom-packing}
    \sum_{(Q,J)\in\mathscr P_{s,n,\iota}}
       \|a_{Q,J}^{(s)}\|_{L^1}
    \lesssim_d
    2^{-2n}\|b\|_{\dot{\mathcal X}_1(\mathcal Q)}|Q_0|.
\end{equation}

For every \(x\),
\[
\begin{aligned}
    \mathcal A_{s,n,\iota}^{\Omega_s^{\mathrm b}}b(x)
    &\leq
    \sum_{(Q,J)\in\mathscr P_{s,n,\iota}}
       \bigl|
         T_{s_J}^{\Omega_s^{\mathrm b}}
         a_{Q,J}^{(s)}(x)
       \bigr|
       \mathbf1_{J^\iota}(x).
\end{aligned}
\]
We next record the single-atom estimate used to sum this expression.
For \((Q,J)\in\mathscr P_{s,n,\iota}\), Fubini's theorem
and H\"older's inequality give
\begin{equation}\label{eq:bad-single-atom}
\begin{aligned}
    \int_{\mathbb R^d}
       \bigl|
         T_{s_J}^{\Omega_s^{\mathrm b}}
         a_{Q,J}^{(s)}(x)
       \bigr|
       |h(x)|\,dx\lesssim_d
    \|\Omega_s^{\mathrm b}\|_{L^q(S^{d-1})}
    \|a_{Q,J}^{(s)}\|_{L^1}
    \inf_{z\in\widehat Q}M_{q'}h(z).
\end{aligned}
\end{equation}
Indeed, the kernel size and annular support give
\[
    \|K_{s_J}^{\Omega_s^{\mathrm b}}\|_{L^q(\mathbb R^d)}
    \lesssim_d
    2^{-s_Jd/q'}
    \|\Omega_s^{\mathrm b}\|_{L^q(S^{d-1})}.
\]
For $y\in Q$ and $z\in\widehat Q$, the support of
$K_{s_J}^{\Omega_s^{\mathrm b}}(\,\cdot-y)$ is contained
in a ball $B(z,C2^{s_J})$. Hence H\"older's inequality yields
\[
\begin{aligned}
    &\int_{\mathbb R^d}
       |K_{s_J}^{\Omega_s^{\mathrm b}}(x-y)|\,|h(x)|\,dx\\
    &\quad\lesssim_d
       2^{-s_Jd/q'}
       \|\Omega_s^{\mathrm b}\|_{L^q(S^{d-1})}
       \|h\|_{L^{q'}(B(z,C2^{s_J}))}\\
    &\quad\lesssim_d
       \|\Omega_s^{\mathrm b}\|_{L^q(S^{d-1})}
       M_{q'}h(z).
\end{aligned}
\]
Integrating against $|a_{Q,J}^{(s)}(y)|$ and taking
the infimum over $z\in\widehat Q$ proves
\eqref{eq:bad-single-atom}.

Since \(p\geq q'\), monotonicity of maximal averages yields
\[
    M_{q'}h\leq M_ph.
\]
It follows from the definition of \(\mathcal Y_p(\mathcal Q)\) that
\[
    \inf_{z\in\widehat Q}M_{q'}h(z)
    \leq
    \|h\|_{\mathcal Y_p(\mathcal Q)}.
\]
Combining this observation with
\eqref{eq:bad-single-atom} and
\eqref{eq:bad-atom-packing}, dividing by \(|Q_0|\), and summing in
\(n\) and \(\iota\), we obtain
\begin{equation}\label{eq:bad-before-tail-corrected}
\begin{aligned}
    \bigl|\mathcal Q_{\mathrm b}(b,h)\bigr|
    &\lesssim_d
    \|b\|_{\dot{\mathcal X}_1(\mathcal Q)}\|h\|_{\mathcal Y_p(\mathcal Q)}
    \sum_{s\geq1}
       \|\Omega_s^{\mathrm b}\|_{L^q}\\
    &\lesssim_d
    \|b\|_{\dot{\mathcal X}_1(\mathcal Q)}\|h\|_{\mathcal Y_p(\mathcal Q)}
    \sum_{s\geq1}
       \|\Omega
         \mathbf1_{\{|\Omega|>\lambda_s\}}\|_{L^q}.
\end{aligned}
\end{equation}

It remains to sum the angular tails.  For every \(\lambda>0\), the
layer-cake representation and Minkowski's inequality imply
\begin{align*}
    \|\Omega\mathbf1_{\{|\Omega|>\lambda\}}\|_{L^q}
    &\leq
    \lambda\mu_\Omega(\lambda)^{1/q}
    +
    \int_\lambda^\infty
       \mu_\Omega(t)^{1/q}\,dt
    \leq
    2\int_{\lambda/2}^\infty
       \mu_\Omega(t)^{1/q}\,dt.
\end{align*}
The last inequality follows from the monotonicity of
\(\mu_\Omega\), since
\[
    \lambda\mu_\Omega(\lambda)^{1/q}
    \leq
    2\int_{\lambda/2}^{\lambda}
       \mu_\Omega(t)^{1/q}\,dt.
\]
Tonelli's theorem now gives
\begin{equation}\label{eq:tail-sum-corrected}
\begin{aligned}
    \sum_{s\geq1}
       \|\Omega
         \mathbf1_{\{|\Omega|>\lambda_s\}}\|_{L^q}
    &\lesssim
    \int_0^\infty
       \#\left\{
         s\geq1:
         L_q(\Omega)2^{\varepsilon s}<2t
       \right\}
       \mu_\Omega(t)^{1/q}\,dt\\
    &\lesssim
    \varepsilon^{-1}
    \int_0^\infty
       \log\left(e+\frac{t}{L_q(\Omega)}\right)
       \mu_\Omega(t)^{1/q}\,dt\\
    &\lesssim_d
    \frac{p_{\sharp}}{q}
    \|\Omega\|_{L^{q,1}\log L(S^{d-1})}.
\end{aligned}    
\end{equation}

The range \(p\geq q'\) is equivalent to \(p' \leq q\).  Therefore
\[
    \frac{p_{\sharp}}{q}
    =
    \frac{\max\{2,p'\}}q
    \leq2,
\]
where for \(q\geq2\) we use \(p_{\sharp}\leq q\), and for
\(1<q<2\) we use \(p_{\sharp}=2\).  Thus
\eqref{eq:bad-before-tail-corrected} and
\eqref{eq:tail-sum-corrected} yield
\begin{equation}\label{eq:bad-final-corrected}
    \bigl|\mathcal Q_{\mathrm b}(b,h)\bigr|
    \lesssim_d
    \|\Omega\|_{L^{q,1}\log L(S^{d-1})}
    \|b\|_{\dot{\mathcal X}_1(\mathcal Q)}\|h\|_{\mathcal Y_p(\mathcal Q)}.
\end{equation}
In particular, the bad angular part introduces no additional
\(p'\)-dependent factor.

Combining
\eqref{eq:good-part-P3-corrected} and
\eqref{eq:bad-final-corrected} proves
\eqref{eq:first-localized-q-final}.
\end{proof}

\subsection{The adjoint localized testing estimate}
\label{subsec:adjoint-LqlogL}
The second testing condition requires only $L^q$ angular
integrability. We use the classical maximal estimate
\begin{equation}\label{eq:classical-angular-Lq}
    \|T_{\Phi,\ast}\|_{L^r\to L^r}
    \leq C_{d,q,r}\|\Phi\|_{L^q(S^{d-1})},
    \qquad 1<q,r<\infty,\quad\int\Phi\,d\sigma=0;
\end{equation}
see \cite{MR837527}. This estimate applies to the dyadic pieces defined using
the smooth radial cutoff above, and to their maximal
two-sided truncations.

\begin{cor}[Second testing estimate]
\label{cor:second-localized-testing}
Let \(1<q<\infty\), let \(\Omega\in L^q(S^{d-1})\)
have mean zero, and let \(q'\leq p<\infty\).
For every \(h\in\mathcal Y_\infty(\mathcal Q)\) and
\(b\in\dot{\mathcal X}_p(\mathcal Q)\), the following
estimate holds uniformly over stopping collections
and bounded measurable integer-valued truncation
functions \(t_1\leq t_2\):
\begin{equation}\label{eq:hb-final}
    |\mathcal Q_{\Omega,t_1}^{t_2}(h,b)|
    \leq C_{d,q}\|\Omega\|_{L^q}
       \|h\|_{\mathcal Y_\infty(\mathcal Q)}
       \|b\|_{\dot{\mathcal X}_p(\mathcal Q)}.
\end{equation}
\end{cor}
\begin{proof}
First assume that \(\Omega\) is bounded.
Applying \cite[Lemma~4.1]{MR4245601} with \(r=q\), together with
\eqref{eq:classical-angular-Lq}, gives
\[
\begin{aligned}
    |\mathcal Q_{\Omega,t_1}^{t_2}(h,b)|
    &\lesssim_d
    \|T_{\Omega,\ast}\|_{L^q\to L^q}
    \|h\|_{\mathcal Y_q(\mathcal Q)}
    \|b\|_{\mathcal Y_{q'}(\mathcal Q)}\\
    &\lesssim_{d,q}
    \|\Omega\|_{L^q(S^{d-1})}
    \|h\|_{\mathcal Y_\infty(\mathcal Q)}
    \|b\|_{\dot{\mathcal X}_p(\mathcal Q)}.
\end{aligned}
\]
Here we used \(q'\leq p\), the monotonicity of the localized
norms, and the fact that \(\dot{\mathcal X}_p(\mathcal Q)\)
inherits its norm from \(\mathcal Y_p(\mathcal Q)\).
The constant is independent of \(p\geq q'\).

For general $\Omega\in L^q$, let
$\Omega_N=\Omega\mathbf1_{\{|\Omega|\leq N\}}
-\int\Omega\mathbf1_{\{|\Omega|\leq N\}}\,d\sigma$.
Then $\Omega_N\to\Omega$ in $L^q$ and in $L^1$.
For fixed bounded truncation functions, the localized forms contain
only finitely many kernel scales. At each scale, the $L^1$ norm
of the kernel difference is at most $C_d\|\Omega_N-\Omega\|_1$.
In the subtracted part, the stopping inputs are disjoint, so the
sum of their absolute convolutions is bounded by the convolution
with $|h|\mathbf1_{Q_0}$. Hence the difference of the localized
pairings is at most
\[
    \frac{C_dN_t}{|Q_0|}\|\Omega_N-\Omega\|_1
       \|h\|_\infty\|b\|_1\longrightarrow0,
\]
where $N_t<\infty$ is the number of possible scales for these
fixed truncation functions. Passing to the limit proves the claim.
\end{proof}

\subsection{Sparse domination for the dyadic maximal truncation}
\label{subsec:completion-LqlogL}

We now combine the two localized testing estimates with
Theorem~\ref{thm3.6} and then remove the boundedness assumption on
the angular kernel.

\begin{thm}[Sparse bound for the dyadic maximal truncation]
\label{thm:dyadic-LqlogL-sparse}
Let \(1<q<\infty\), let
\(\Omega\in L^{q,1}\log L(S^{d-1})\) satisfy
\eqref{eq:Omega-global-cancellation}, and let \(q'\leq p<\infty\).  Then the
dyadic maximal truncation associated with the kernel family
\(\{K_k^\Omega\}_{k\in\mathbb Z}\) satisfies
\begin{equation}\label{eq:dyadic-LqlogL-sparse}
    \|T_{\Omega,\ast}\|_{(1,p)\text{-}\mathrm{sparse}}
    \lesssim_{d,q}
    p'
    \|\Omega\|_{L^{q,1}\log L(S^{d-1})}.
\end{equation}
\end{thm}

\begin{proof}
We first prove the global bound needed for
Theorem~\ref{thm3.6}.
Let $\Theta\in L^{q,1}\log L(S^{d-1})$ have mean zero.
By the embedding $L^{q,1}(S^{d-1})\hookrightarrow L^q(S^{d-1})$
and \eqref{eq:classical-angular-Lq} with $r=2$, we have
\[
    \|T_{\Theta,\ast}\|_{L^2\to L^2}
    \lesssim_{d,q}
    \|\Theta\|_{L^q(S^{d-1})}
    \lesssim_q L_q(\Theta)
    \leq
    \|\Theta\|_{L^{q,1}\log L(S^{d-1})}.
\]
If, in addition, \(\Theta\) is bounded, then the kernel family
\(\{K_k^\Theta\}\) satisfies
\eqref{equ3.1}, while
Proposition~\ref{prop:first-localized-LqlogL} and
Corollary~\ref{cor:second-localized-testing} give
\[
    C_L[K^\Theta](1,p)
    \lesssim_{d,q}
    \bigl(p'+1\bigr)\|\Theta\|_{L^{q,1}\log L(S^{d-1})}.
\]
Consequently, Theorem~\ref{thm3.6}, with \(r=2\), yields
\[
    \|T_{\Theta,\ast}\|_{(1,p)\text{-}\mathrm{sparse}}
    \lesssim_{d,q}
    p'\|\Theta\|_{L^{q,1}\log L(S^{d-1})}.
\]
Thus the explicit factor $p'$ comes from the first
testing condition; for fixed $q$, the second testing
constant is independent of $p\geq q'$.

It remains to remove the boundedness assumption.  For \(N\geq1\), set
\[
    d_N:=\int_{S^{d-1}}
       \Omega\mathbf1_{\{|\Omega|\leq N\}}\,d\sigma,
    \qquad
    \widetilde\Omega_N
    :=\Omega\mathbf1_{\{|\Omega|\leq N\}}-d_N.
\]
Then \(\widetilde\Omega_N\) is bounded and has mean zero.  Since
\(\int_{S^{d-1}}\Omega\,d\sigma=0\), we have
\[
    d_N=-\int_{S^{d-1}}
       \Omega\mathbf1_{\{|\Omega|>N\}}\,d\sigma.
\]
The truncation property recorded after
\eqref{eq:Lq1logL-functional} therefore gives
\[
    \sup_{N\geq1}\|\widetilde\Omega_N\|_{L^{q,1}\log L(S^{d-1})}
    \lesssim_q\|\Omega\|_{L^{q,1}\log L(S^{d-1})},
    \qquad
    \|\Omega-\widetilde\Omega_N\|_{L^{q,1}\log L(S^{d-1})}\longrightarrow0.
\]
Applying the bounded-kernel estimate to $\widetilde\Omega_N$ yields a
sparse bound whose constant is uniform in $N$. More precisely, for
bounded, compactly supported functions $f$ and $g$,
\[
    \bigl|
        \langle T_{\widetilde\Omega_N,\ast}f,g\rangle
    \bigr|
    \lesssim_{d,q}
    p' \|\Omega\|_{L^{q,1}\log L(S^{d-1})}
    \sup_{\mathcal S}
    \Lambda_{\mathcal S,1,p}(f,g),
\]
where the supremum is taken over all $1/2$-sparse collections
$\mathcal S$.

On the other hand, sublinearity of the maximal truncation gives the
pointwise estimate
\[
    \bigl|
        T_{\widetilde\Omega_N,\ast}f
        -T_{\Omega,\ast}f
    \bigr|
    \leq
    T_{\Omega-\widetilde\Omega_N,\ast}f.
\]
Therefore, by the $L^2$ estimate established above and the approximation
property of $\widetilde\Omega_N$,
\[
    \bigl\|
        T_{\widetilde\Omega_N,\ast}f
        -T_{\Omega,\ast}f
    \bigr\|_{L^2}
    \leq
    \bigl\|
        T_{\Omega-\widetilde\Omega_N,\ast}f
    \bigr\|_{L^2}
    \lesssim_{d,q}
    \|\Omega-\widetilde\Omega_N\|_{L^{q,1}\log L(S^{d-1})}
    \|f\|_{L^2}
    \longrightarrow 0.
\]
Pairing with $g$ and passing to the limit, we obtain
\[
    \bigl|
        \langle T_{\Omega,\ast}f,g\rangle
    \bigr|
    \lesssim_{d,q}
    p' \|\Omega\|_{L^{q,1}\log L(S^{d-1})}
    \sup_{\mathcal S}
    \Lambda_{\mathcal S,1,p}(f,g).
\]
Using \eqref{eq:sparse-supremum-finite} and choosing a sparse
collection whose form is within a factor \(2\) of the supremum, as
in the bounded-kernel case, the preceding estimate proves
\eqref{eq:dyadic-LqlogL-sparse}.
\end{proof}

\subsection{Passage to the standard maximal truncation}
\label{subsec:standard-maximal-truncation}

It remains to pass from the dyadic maximal truncation to the
standard radial truncation.  This comparison produces the rough
maximal term \(M_\Omega\).  We first establish the required sparse
bound for \(M_\Omega\), for which no cancellation assumption on
\(\Omega\) is needed.
\begin{prop}[Sparse domination of the rough maximal operator]
\label{prop:rough-maximal-sparse}
Let \(1<q<\infty\), let
\(\Omega\in L^{q,1}\log L(S^{d-1})\), and let \(q'\leq p<\infty\).
No cancellation assumption on \(\Omega\) is required.  Then
\begin{equation}\label{eq:rough-maximal-sparse}
    \|M_\Omega\|_{(1,p)\text{-}\mathrm{sparse}}
    \lesssim_{d,q}
    p'
    \|\Omega\|_{L^{q,1}\log L(S^{d-1})},
\end{equation}
where
\[
    M_\Omega f(x)
    :=
    \sup_{r>0}\frac1{r^d}
    \int_{|y|<r}
       |\Omega(y/|y|)|\,|f(x-y)|\,dy.
\]
\end{prop}

\begin{proof}
Let
\[
    \eta(z):=\sum_{u=2}^{3}\psi(2^{-u}z).
\]
By the support and partition-of-unity properties of \(\psi\),
\[
    \eta(z)=1
    \qquad\text{whenever}\qquad
    \frac12\leq |z|\leq1.
\]
For an angular function \(\Phi\), define
\[
    \mathcal A_j^\Phi f(x)
    :=
    \int_{\mathbb R^d}
       \eta(2^{-j}y)
       \frac{\Phi(y/|y|)}{|y|^d}
       f(x-y)\,dy.
\]

Fix \(r>0\), and choose \(k\in\mathbb Z\) such that
\(2^{k-1}<r\leq2^k\).  
Decompose \(B(0,r)\) into the annuli
\[
    \mathcal R_{k-s}
    :=
    \{y:2^{k-s-1}<|y|\leq2^{k-s}\},
   \qquad s\geq0.
\]
Since
\[
    B(0,r)
    \subset
    \bigcup_{s=0}^{\infty}\mathcal R_{k-s},
\]
and the integrand defining $M_\Omega$ is nonnegative, we may enlarge
each intersection $B(0,r)\cap\mathcal R_{k-s}$ to the full annulus
$\mathcal R_{k-s}$. Indeed, $r^{-d}\leq 2^d2^{-kd}$, while
$|y|^d\leq 2^{(k-s)d}$ for $y\in \mathcal R_{k-s}$. Therefore,
\begin{align}
   M_\Omega f(x)
&\lesssim_d
\sup_{k\in\mathbb Z}
2^{-kd}\sum_{s=0}^{\infty}
2^{(k-s)d}\mathcal A_{k-s}^{|\Omega|}(|f|)(x)
\nonumber\\
&=
\sup_{k\in\mathbb Z}
\sum_{s=0}^{\infty}
2^{-sd}\mathcal A_{k-s}^{|\Omega|}(|f|)(x)
\lesssim_d
\sup_{j\in\mathbb Z}\mathcal A_j^{|\Omega|}(|f|)(x).
\label{eq:rough-maximal-annular-reduction}
\end{align}

Set
\[
    a_\Omega
    :=
    \int_{S^{d-1}}|\Omega(\theta)|\,d\sigma(\theta),
    \qquad
    \Theta:=|\Omega|-a_\Omega.
\]
Then \(\int_{S^{d-1}}\Theta\,d\sigma=0\).  Moreover, since
\(\sigma(S^{d-1})=1\), the layer-cake formula gives
\[
    a_\Omega
    \leq\frac1qL_q(\Omega)
    \leq\frac1q\|\Omega\|_{L^{q,1}\log L}.
\]
Using the equivalence with the usual homogeneous
Lorentz-Zygmund norm and
\[\|a_\Omega\mathbf1_{S^{d-1}}\|_{L^{q,1}\log L}
\lesssim_q a_\Omega,\] we obtain
\[
    \|\Theta\|_{L^{q,1}\log L}
    +a_\Omega
    \lesssim_q
    \|\Omega\|_{L^{q,1}\log L}.
\]
Since
\[
    \mathcal A_j^{|\Omega|}(|f|)
    =
    \mathcal A_j^\Theta(|f|)
    +a_\Omega\mathcal A_j^1(|f|),
\]
and
\[
    \mathcal A_j^\Theta f
    =
    K_{j+2}^\Theta*f+K_{j+3}^\Theta*f,
\]
the two consecutive kernel pieces on the right are included in the
dyadic maximal truncation.  Hence
\[
    \sup_{j\in\mathbb Z}
       |\mathcal A_j^\Theta(|f|)|
    \leq
    T_{\Theta,\ast}(|f|).
\]
The constant-angular part satisfies
\[
    \sup_{j\in\mathbb Z}
       \mathcal A_j^1(|f|)
    \lesssim_d Mf.
\]
Consequently, \eqref{eq:rough-maximal-annular-reduction} yields
\[
    M_\Omega f
    \lesssim_d
    T_{\Theta,\ast}(|f|)
    +a_\Omega Mf.
\]
Let \(f\) and \(g\) be bounded and compactly supported.
Applying Theorem~\ref{thm:dyadic-LqlogL-sparse} to the mean-zero
function \(\Theta\), and using the standard sparse bound for the
Hardy--Littlewood maximal operator, we obtain
\begin{align*}
    \bigl|\langle M_\Omega f,g\rangle\bigr|
    &\lesssim_d
    \bigl\langle T_{\Theta,\ast}(|f|),|g|\bigr\rangle
    +
    a_\Omega\langle M|f|,|g|\rangle
    \\
    &\lesssim_{d,q}
    \left(
        p'\|\Theta\|_{L^{q,1}\log L}
        +a_\Omega
    \right)
    \sup_{\mathcal S}
    \Lambda_{\mathcal S,1,p}(f,g)
    \\
    &\lesssim_{d,q}
    p'
    \|\Omega\|_{L^{q,1}\log L}
    \sup_{\mathcal S}
    \Lambda_{\mathcal S,1,p}(f,g),
\end{align*}
where the supremum is taken over all \(1/2\)-sparse collections.
Using \eqref{eq:sparse-supremum-finite} and the same near-maximizing
choice of a sparse collection \(\mathcal S\), the preceding estimate
gives
\[
    \bigl|\langle M_\Omega f,g\rangle\bigr|
    \lesssim_{d,q}
    p'
    \|\Omega\|_{L^{q,1}\log L(S^{d-1})}
    \Lambda_{\mathcal S,1,p}(f,g).
\]
This proves \eqref{eq:rough-maximal-sparse}.
\end{proof}

With the preceding auxiliary estimate in hand, we can now complete
the passage to the standard maximal truncation.

\begin{thm}[Sparse bound for the standard maximal truncation]
\label{thm:continuous-LqlogL-sparse}
Let \(1<q<\infty\), let
\(\Omega\in L^{q,1}\log L(S^{d-1})\) satisfy
\eqref{eq:Omega-global-cancellation}, and let \(q'\leq p<\infty\).  Then the
standard radial maximal truncation satisfies
\begin{equation}\label{eq:continuous-LqlogL-sparse}
    \|T_\Omega^\ast\|_{(1,p)\text{-}\mathrm{sparse}}
    \lesssim_{d,q}
    p'
    \|\Omega\|_{L^{q,1}\log L(S^{d-1})}.
\end{equation}
\end{thm}
\begin{proof}
Let $f$ and $g$ be bounded and compactly supported. By the first comparison in
Section~\ref{subsec:dyadic-abstract-reduction}, which remains valid
for every \(\Omega\in L^1(S^{d-1})\), we have
\[
    T_\Omega^\ast f
    \lesssim_d
    T_{\Omega,\ast}f+M_\Omega f.
\]
Consequently, Theorem~\ref{thm:dyadic-LqlogL-sparse}
and Proposition~\ref{prop:rough-maximal-sparse} imply
\begin{align*}
    \bigl|\langle T_\Omega^\ast f,g\rangle\bigr|
    &\lesssim_d
    \bigl\langle T_{\Omega,\ast}f,|g|\bigr\rangle
    +
    \langle M_\Omega f,|g|\rangle
    \lesssim_{d,q}
    p'
    \|\Omega\|_{L^{q,1}\log L(S^{d-1})}
    \sup_{\mathcal S}
    \Lambda_{\mathcal S,1,p}(f,g),
\end{align*}
where the supremum is taken over all $1/2$-sparse collections.
Using \eqref{eq:sparse-supremum-finite} and the same near-maximizing
selection as above, the preceding estimate yields
\eqref{eq:continuous-LqlogL-sparse}.
\end{proof}

\section*{Acknowledgments}
The author is deeply grateful to Professor Kangwei Li for his
continued guidance, many helpful discussions, and careful comments
on earlier versions of this manuscript.

\section*{Declaration of Generative AI Use}
During the preparation of this manuscript, the author used ChatGPT (OpenAI) 
to improve the English presentation and expository clarity and suggest clearer 
or more complete formulations of a limited number of routine technical steps. 
The main results, proof strategy, and core mathematical arguments were developed 
independently by the author. All AI-assisted suggestions were critically assessed, 
independently verified, and revised as needed by the author, who takes full 
responsibility for the content of the manuscript.

\bibliographystyle{plain}
\bibliography{ams.bib}
\end{document}